\documentclass[12pt]{article}
\usepackage{e-jc}

\usepackage{amsthm,amsmath,amssymb}

\usepackage{graphicx}

\usepackage[colorlinks=true,citecolor=black,linkcolor=black,urlcolor=blue]{hyperref}

\usepackage[table]{xcolor}

\theoremstyle{plain}
\newtheorem{theorem}{Theorem}
\newtheorem{thm}[theorem]{Theorem}
\newtheorem{lemma}[theorem]{Lemma}
\newtheorem{cor}[theorem]{Corollary}
\newtheorem{prop}[theorem]{Proposition}
\newtheorem{exam}{Example}
\newtheorem{ques}{Question}

\theoremstyle{remark}

\newcommand{\dd}{difference diagonal}
\newcommand{\sd}{sum diagonal}
\newcommand{\q}[2]{\mbox{$Q_{#1\times#2}$}}
\newcommand{\qq}{\mbox{$Q_{m \times n}$}}
\newcommand{\gamqq}{\mbox{$\gamma(\qq)$}}
\newcommand{\beem}[1]{\mbox{$\ ( \bmod\  #1)$}}

\newcommand{\xy}{\mbox{($x$,$y$)}}

\newcommand{\vsm}{\vspace{.1in}}

\newcommand{\mds}{minimum dominating set}
\newcommand{\zah}{\mathbb{Z}}

\title{\bf Domination of the \\ rectangular queen's graph}

\author{S\'andor Boz\'oki\\
\small Institute for Computer Science and \\[-0.8ex]
\small Control (SZTAKI), Budapest, Hungary;\\[-0.8ex]
\small Corvinus University of Budapest\\[-0.8ex]
\small\tt bozoki.sandor@sztaki.hu\\
\and
P\'eter G\'al\\
\small Kofax-Recognita Zrt., Budapest, Hungary\\[-0.8ex]
\small\tt galpetya@gmail.com\\
\and
Istv\'an Marosi\\
\small Kofax-Recognita Zrt., Budapest, Hungary\\[-0.8ex]
\small\tt imarosi@gmail.com\\
\and
William D. Weakley\\
\small Purdue University Fort Wayne, USA\\[-0.8ex]
\small\tt weakley@pfw.edu\\
}

\date{\vspace{-5ex}}
\begin{document}

\maketitle


\begin{abstract}
  The queens graph $Q_{m \times n}$ has the squares of the $m \times n$ chessboard as its vertices;
  two squares are adjacent if they are in
  the same row, column, or diagonal of the board.  A set $D$  of squares of $Q_{m \times n}$ is a \emph{dominating set}
  for $Q_{m \times n}$ if every square of $Q_{m \times n}$ is either in $D$
  or adjacent to a square in $D$. The minimum size of a dominating set
  of $Q_{m \times n}$ is the \emph{domination number}, denoted by $\gamma(Q_{m \times n})$.

Values of $\gamma(Q_{m \times n}), \, 4 \leq m  \leq n \leq  18,
\,$ are given here, in each case with a f{\kern0pt}ile of minimum
dominating sets (often all of them, up to symmetry) in an online
\href{https://www.combinatorics.org/ojs/index.php/eljc/article/view/v26i4p45/HTML}{appendix}.
In these ranges for $m$ and $n$, monotonicity fails once:
$\gamma(Q_{8 \times 11}) = 6 > 5 = \gamma(Q_{9 \times 11})
 = \gamma(Q_{10 \times 11}) = \gamma(Q_{11 \times 11})$.

Let $g(m)$ [respectively $g^{*}(m)$] be the largest integer such that $m$ queens suf{\kern0pt}f{\kern0pt}ice
to dominate the $(m+1) \times g(m)$ board [respectively, to dominate the $(m+1) \times g^{*}(m)$
 board with no two queens in a row].
Starting from the elementary bound $g(m) \leq 3m$, domination when the board is far
from square is investigated.  It is shown (Theorem \ref{by3k}) that $g(m) = 3m$ can only occur when
$m \equiv 0, 1, 2, 3, \mbox{or } 4 \mbox{ (mod 9)}$, with an online appendix showing that
this does occur for $m \leq 40, m \neq 3$. Also (Theorem \ref{upperbounds}),
if $m \equiv 5, 6, \mbox{or } 7 \mbox{ (mod 9)}$ then $g^{*}(m) \leq 3m-2$,
and if $m \equiv 8 \mbox{ (mod 9)}$ then $g^{*}(m) \leq 3m-4$.   It is shown that equality holds in these
bounds for $m \leq 40 $.

Lower bounds on $\gamma(Q_{m \times n})$ are given.
In particular, if $m \leq n$   then $\gamma(Q_{m \times n}) \geq \min \{ m, \lceil (m+n-2)/4 \rceil \}$.

Two types of dominating sets (orthodox covers and centrally strong sets) are developed; each type is
shown to give good upper bounds of $\gamma(Q_{m \times n})$ in several cases.

Three questions are posed: whether monotonicity of $\gamma(\q{m}{n})$ holds (other than
from $(m, n) = (8, 11)$ to $(9, 11)$), whether $\gamma(\q{m}{n}) = (m+n-2)/4$ occurs with $m \leq n < 3m+2$
(other than for $(m, n) = (3, 3)$ and $(11, 11)$), and whether the lower
bound given above can be improved.

 A set of squares is \emph{independent} if no two of its squares are adjacent.
 The minimum size of an independent dominating set of $Q_{m \times n}$ is the \emph{independent domination number},
denoted by $i(Q_{m \times n})$. Values of $i(Q_{m \times n}), \, 4 \leq m  \leq n \leq  18, \,$
 are given  here,  in each case with some  minimum dominating sets.
In these ranges for $m$ and $n$, monotonicity fails twice:
$i(Q_{8 \times 11}) = 6 > 5 = i(Q_{9 \times 11}) = i(Q_{10 \times 11}) = i(Q_{11 \times 11})$,
and $i(Q_{11 \times 18}) = 9 > 8 = i(Q_{12 \times 18})$.

\end{abstract}

\section{Introduction}

Let $m$ and $n$ be positive integers. We will identify the $m \times n$ chessboard with a rectangle in the
Cartesian plane, having sides parallel to the coordinate axes.  We place the board so that the center of every
square has integer coordinates, and refer to each square by the coordinates \xy\ of its center.  Unless
otherwise noted, squares have edge length one, and the board is placed so that the lower left corner
has center $(1, 1)$; sometimes it is more convenient to use squares of edge length two
or to place the board with its center at the origin of the coordinate system.  By symmetry it suf{\kern0pt}f{\kern0pt}ices to
consider the case $m \leq n$, which we will assume throughout: the board has at least as many columns as rows.

The square \xy\ is in
{\em column} $x$ and {\em row} $y$.
Columns and rows will be referred to
collectively as {\em orthogonals}. The
{\em dif{\kern0pt}ference diagonal\/} (respectively
{\em sum diagonal\/}) through square \xy\
is the set of all board squares with centers on the
line of slope +1 (respectively $-1$) through the
point \xy. The value of $y-x$ is the same for each square \xy\
on a dif{\kern0pt}ference diagonal, and we will refer to the diagonal by this value.
Similarly, the value of $y+x$ is the same for each
square on a sum diagonal, and we associate this value to the diagonal.
Orthogonals and diagonals are collectively referred to as \!\!\emph{ lines} of the board.
\vsm

The queens graph \qq\ has the squares of the $m \times n$ chessboard as its vertices; two squares are
adjacent if they are in some line of \qq. A set $D$ of squares of \qq\ is a {\em dominating set}
for \qq\ if every square of \qq\ is either in $D$ or adjacent to a square in $D$. The minimum size of a
dominating set is the \emph{domination number}, denoted by $\gamma(\qq)$.
A set of squares is {\em independent} if no two squares in the set are adjacent.  \vsm

Almost all previous work on queen domination has concerned square boards. The problem of f{\kern0pt}inding
values of $\gamma(\q{n}{n})$ has interested mathematicians for over 150 years. The f{\kern0pt}irst published work
is that of De Jaenisch \cite{JA} in 1862, who gave minimum dominating sets and minimum independent
dominating sets of $Q_{n \times n}$ for $n \leq 8$.  His work was brief{\kern0pt}ly summarized by Rouse Ball \cite{RB} in 1892,
who considered several other questions about queen domination. In 1901, W. Ahrens \cite[Chapter X]{AH1} gave
minimum dominating sets for \q{9}{9}, and in 1902-3, K. von Szily \cite{V1, V2} gave minimum dominating sets
of \q{n}{n} for $10 \leq n \leq 13$ and $n = 17$.  Proof that these sets were
minimum had to wait for later work, described below. De Jaenisch, Ahrens, and von
Szily also worked extensively to f{\kern0pt}ind the number of dif{\kern0pt}ferent minimum dominating sets for each $n$, often
giving lists with one set from each symmetry class. Many of these results were collected by Ahrens in
the 1910 edition \cite{AH} of his book, and can also be found in its later editions.

 More detail and some examples from recent work on domination of $Q_{n \times n}$ can be found in
\cite{Weakley2018}.

The f{\kern0pt}irst published work on nonsquare boards of which we are aware is in
Watkins \cite{Watkins2004}:  the values  $\gamma(Q_{5 \times 12}) = 4$
and $\gamma(Q_{6 \times 10}) = 4$
(see Problem 8.4 on p.~132 and Figure 8.19 on p.~137), found by D.~C.~Fisher. \vsm

Say that two \mds s of \gamqq\ are
\emph{equivalent} if there
is an isometry of the $m \times n$ chessboard that carries one to the other.

We have computed \gamqq\ for  rectangular chessboards with $4 \leq m \leq n \leq 18$. \linebreak
Results are given in Table 1; for most $m$ and $n$ we give a f{\kern0pt}ile of minimum dominating \linebreak
sets with  one from every equivalence class, unless the number of equivalence classes is large.
The online appendix at \\
\url{https://www.combinatorics.org/ojs/index.php/eljc/article/view/v26i4p45/HTML} \linebreak
includes the computational results. For each set, we describe its symmetry and say whether it can be
obtained by one of the constructions in Section \ref{construct}.

\begin{center}
\begin{tabular}{|c||c|c|c|c|c|c|c|c|c|c|c|c|c|c|c|}
\hline
 $n \diagdown m$ & \hspace*{0.1mm} 4 \hspace*{0.1mm}
                 & \hspace*{0.1mm} 5 \hspace*{0.1mm}
                 & \hspace*{0.1mm} 6 \hspace*{0.1mm}
                 & \hspace*{0.1mm} 7 \hspace*{0.1mm}
                 & \hspace*{0.1mm} 8 \hspace*{0.1mm}
                 & \hspace*{0.1mm} 9 \hspace*{0.1mm} & 10 & 11 & 12 & 13 & 14 & 15 & 16 & 17 & 18  \\
\hline
\hline
               4 &  \href{https://www.combinatorics.org/files/v26i4p45/04x04_2Q.html}{\emph{2}} & & & & & & & & & & & & & &   \\
\hline
               5 &  \href{https://www.combinatorics.org/files/v26i4p45/04x05_2Q.html}{2} &
                    \href{https://www.combinatorics.org/files/v26i4p45/05x05_3Q.html}{\emph{3}} & & & & & & & & & & & & &  \\
\hline
               6 &  \href{https://www.combinatorics.org/files/v26i4p45/04x06_3Q.html}{3} &
                    \href{https://www.combinatorics.org/files/v26i4p45/05x06_3Q.html}{3} &
                    \href{https://www.combinatorics.org/files/v26i4p45/06x06_3Q.html}{\emph{3}} & & & & & & & & & & & & \\
\hline
               7 &  \href{https://www.combinatorics.org/files/v26i4p45/04x07_3Q.html}{3} &
                    \href{https://www.combinatorics.org/files/v26i4p45/05x07_3Q.html}{3} &
                    \href{https://www.combinatorics.org/files/v26i4p45/06x07_4Q.html}{4} &
                    \href{https://www.combinatorics.org/files/v26i4p45/07x07_4Q.html}{\emph{4}} & & & & & & & & & & & \\
\hline
               8 &  \href{https://www.combinatorics.org/files/v26i4p45/04x08_3Q.html}{3} &
                    \href{https://www.combinatorics.org/files/v26i4p45/05x08_4Q.html}{4} &
                    \href{https://www.combinatorics.org/files/v26i4p45/06x08_4Q.html}{4} &
                    \href{https://www.combinatorics.org/files/v26i4p45/07x08_5Q.html}{5} &
                    \href{https://www.combinatorics.org/files/v26i4p45/08x08_5Q.html}{\emph{5}} & & & & & & & & & & \\
\hline
               9 &  {4} &
                    \href{https://www.combinatorics.org/files/v26i4p45/05x09_4Q.html}{4} &
                    \href{https://www.combinatorics.org/files/v26i4p45/06x09_4Q.html}{4} &
                    \href{https://www.combinatorics.org/files/v26i4p45/07x09_5Q.html}{5} &
                    \href{https://www.combinatorics.org/files/v26i4p45/08x09_5Q.html}{5} &
                    \href{https://www.combinatorics.org/files/v26i4p45/09x09_5Q.html}{\emph{5}}  & & & & & & & & & \\
\hline
              10 &  {4} &
                    \href{https://www.combinatorics.org/files/v26i4p45/05x10_4Q.html}{4} &
                    \href{https://www.combinatorics.org/files/v26i4p45/06x10_4Q.html}{4} &
                    \href{https://www.combinatorics.org/files/v26i4p45/07x10_5Q.html}{5} &
                    \href{https://www.combinatorics.org/files/v26i4p45/08x10_5Q.html}{5} &
                    \href{https://www.combinatorics.org/files/v26i4p45/09x10_5Q.html}{5} &
                    \href{https://www.combinatorics.org/files/v26i4p45/10x10_5Q.html}{\emph{5}}  & & & & & & & & \\
\hline
              11 &  {4} &
                    \href{https://www.combinatorics.org/files/v26i4p45/05x11_4Q.html}{4} &
                    \href{https://www.combinatorics.org/files/v26i4p45/06x11_5Q.html}{5} &
                    \href{https://www.combinatorics.org/files/v26i4p45/07x11_5Q.html}{5} &
                    \href{https://www.combinatorics.org/files/v26i4p45/08x11_6Q.html}{\textbf{6}} &
                    \href{https://www.combinatorics.org/files/v26i4p45/09x11_5Q.html}{5} &
                    \href{https://www.combinatorics.org/files/v26i4p45/10x11_5Q.html}{5} &
                    \href{https://www.combinatorics.org/files/v26i4p45/11x11_5Q.html}{\emph{5}}   & & & & & &  &\\
\hline
              12 &  {4} &
                    \href{https://www.combinatorics.org/files/v26i4p45/05x12_4Q.html}{4} &
                    \href{https://www.combinatorics.org/files/v26i4p45/06x12_5Q.html}{5} &
                    \href{https://www.combinatorics.org/files/v26i4p45/07x12_5Q.html}{5} &
                    \href{https://www.combinatorics.org/files/v26i4p45/08x12_6Q.html}{6}      &
                    \href{https://www.combinatorics.org/files/v26i4p45/09x12_6Q.html}{6} &
                    \href{https://www.combinatorics.org/files/v26i4p45/10x12_6Q.html}{6}  &
                    \href{https://www.combinatorics.org/files/v26i4p45/11x12_6Q.html}{6}  &
                    \href{https://www.combinatorics.org/files/v26i4p45/12x12_6Q.html}{\emph{6}}  & & & & & &  \\
\hline
              13 &  4 &
                    {5}  &
                    \href{https://www.combinatorics.org/files/v26i4p45/06x13_5Q.html}{5} &
                    \href{https://www.combinatorics.org/files/v26i4p45/07x13_6Q.html}{6}  &
                    \href{https://www.combinatorics.org/files/v26i4p45/08x13_6Q.html}{6} &
                    \href{https://www.combinatorics.org/files/v26i4p45/09x13_6Q.html}{6} &
                    \href{https://www.combinatorics.org/files/v26i4p45/10x13_7Q.html}{7}  &
                    \href{https://www.combinatorics.org/files/v26i4p45/11x13_7Q.html}{7}  &
                    \href{https://www.combinatorics.org/files/v26i4p45/12x13_7Q.html}{7}   &
                    \href{https://www.combinatorics.org/files/v26i4p45/13x13_7Q.html}{\emph{7}} & & & &  &  \\
\hline
              14 &  4 &
                    {5} &
                    {6} &
                    \href{https://www.combinatorics.org/files/v26i4p45/07x14_6Q.html}{6}  &
                    \href{https://www.combinatorics.org/files/v26i4p45/08x14_6Q.html}{6} &
                    \href{https://www.combinatorics.org/files/v26i4p45/09x14_6Q.html}{6} &
                    \href{https://www.combinatorics.org/files/v26i4p45/10x14_7Q.html}{7}  &
                    \href{https://www.combinatorics.org/files/v26i4p45/11x14_7Q.html}{7}   &
                    \href{https://www.combinatorics.org/files/v26i4p45/12x14_8Q.html}{8}    &
                    \href{https://www.combinatorics.org/files/v26i4p45/13x14_8Q.html}{8}   &
                    \href{https://www.combinatorics.org/files/v26i4p45/14x14_8Q.html}{\emph{8}} & & &  &             \\
\hline
              15 &  4 &
                    {5} &
                    {6} &
                    \href{https://www.combinatorics.org/files/v26i4p45/07x15_6Q.html}{6} &
                    \href{https://www.combinatorics.org/files/v26i4p45/08x15_6Q.html}{6} &
                    \href{https://www.combinatorics.org/files/v26i4p45/09x15_7Q.html}{7} &
                    \href{https://www.combinatorics.org/files/v26i4p45/10x15_7Q.html}{7} &
                    \href{https://www.combinatorics.org/files/v26i4p45/11x15_7Q.html}{7} &
                    \href{https://www.combinatorics.org/files/v26i4p45/12x15_8Q.html}{8} &
                    \href{https://www.combinatorics.org/files/v26i4p45/13x15_8Q.html}{8} &
                    \href{https://www.combinatorics.org/files/v26i4p45/14x15_8Q.html}{8} &
                    \href{https://www.combinatorics.org/files/v26i4p45/15x15_9Q.html}{\emph{9}} & & &             \\
\hline
              16 &  4 &
                    5 &
                    {6} &
                    \href{https://www.combinatorics.org/files/v26i4p45/07x16_6Q.html}{6} &
                    \href{https://www.combinatorics.org/files/v26i4p45/08x16_7Q.html}{7} &
                    \href{https://www.combinatorics.org/files/v26i4p45/09x16_7Q.html}{7} &
                    \href{https://www.combinatorics.org/files/v26i4p45/10x16_7Q.html}{7} &
                    \href{https://www.combinatorics.org/files/v26i4p45/11x16_8Q.html}{8} &
                    \href{https://www.combinatorics.org/files/v26i4p45/12x16_8Q.html}{8} &
                    \href{https://www.combinatorics.org/files/v26i4p45/13x16_8Q.html}{8} &
                    \href{https://www.combinatorics.org/files/v26i4p45/14x16_9Q.html}{9} &
                    \href{https://www.combinatorics.org/files/v26i4p45/15x16_9Q.html}{9} &
                    \href{https://www.combinatorics.org/files/v26i4p45/16x16_9Q.html}{\emph{9}} &  &           \\
\hline
              17 &  4 &
                    5 &
                    6 &
                    {7} &
                    \href{https://www.combinatorics.org/files/v26i4p45/08x17_7Q.html}{7} &
                    \href{https://www.combinatorics.org/files/v26i4p45/09x17_7Q.html}{7}  &
                    \href{https://www.combinatorics.org/files/v26i4p45/10x17_8Q.html}{8} &
                    \href{https://www.combinatorics.org/files/v26i4p45/11x17_8Q.html}{8} &
                    \href{https://www.combinatorics.org/files/v26i4p45/12x17_8Q.html}{8} &
                    \href{https://www.combinatorics.org/files/v26i4p45/13x17_9Q.html}{9} &
                    \href{https://www.combinatorics.org/files/v26i4p45/14x17_9Q.html}{9} &
                    \href{https://www.combinatorics.org/files/v26i4p45/15x17_9Q.html}{9} &
                    \href{https://www.combinatorics.org/files/v26i4p45/16x17_9Q.html}{9} &
                    \href{https://www.combinatorics.org/files/v26i4p45/17x17_9Q.html}{\emph{9}}  &            \\
\hline
              18 &  4 &
                    5 &
                    6 &
                    {7} &
                    \href{https://www.combinatorics.org/files/v26i4p45/08x18_7Q.html}{7} &
                    \href{https://www.combinatorics.org/files/v26i4p45/09x18_8Q.html}{8} &
                    \href{https://www.combinatorics.org/files/v26i4p45/10x18_8Q.html}{8} &
                    \href{https://www.combinatorics.org/files/v26i4p45/11x18_8Q.html}{8} &
                    \href{https://www.combinatorics.org/files/v26i4p45/12x18_8Q.html}{8} &
                    \href{https://www.combinatorics.org/files/v26i4p45/13x18_9Q.html}{9} &
                    \href{https://www.combinatorics.org/files/v26i4p45/14x18_9Q.html}{9} &
                    \href{https://www.combinatorics.org/files/v26i4p45/15x18_9Q.html}{9} &
                    \href{https://www.combinatorics.org/files/v26i4p45/16x18_9Q.html}{9} &
                    \href{https://www.combinatorics.org/files/v26i4p45/17x18_9Q.html}{9} &
                    \href{https://www.combinatorics.org/files/v26i4p45/18x18_9Q.html}{\emph{9}}              \\

\hline
\end{tabular} \\[3mm]
Table 1: Values of $\gamma(Q_{m \times n}), \, 4 \leq m  \leq n \leq  18 $  (\href{http://oeis.org/A274138}{OEIS A274138})
\end{center}

The computation was done with a backtracking algorithm.
The backtrack condition minimizes the number of queens placed.
If a solution is found with $k$ queens, then the
remaining search space is limited to at most $k-1$ queens.
The algorithm places a single queen in a position covering
the top left cell and does a recursive call to cover all remaining cells.
Some heuristics are used also to {f}{i}nd the {f}{i}rst solution faster:
the {f}{i}rst queen is placed in the middle of the board
(actually in the closest to middle position attacking the top left
cell); other possible attacking positions are only tried later.
Frequently this position is part of a minimal solution.

Once it is shown that there is no solution with $k-1$ queens,
a search for other solutions with $k$ queens is made.

Cockayne \cite[Problem 1]{CO} introduced monotonicity
\[
\gamma(Q_{n \times n}) \stackrel{?}{\leq}  \gamma(Q_{(n+1) \times (n+1)})
\]
as an open problem (see also in Chartrand, Haynes, Henning and Zhang \cite[Conjecture 1.2.1 on page 7]{ChartrandHaynesHenningZhang2019}.

A remarkable observation about \q{8}{11}: six queens (with bold
typeface in Table 1) are necessary to dominate it, though f{\kern0pt}ive
queens are suf{\kern0pt}f{\kern0pt}icient (and necessary) to
dominate each of \q{9}{11}, \q{10}{11}, \q{11}{11}. A possible
explanation for this is given later. We note that  f{\kern0pt}ive queens can
cover all but one square of {\q{8}{11}}. One of the
\href{https://www.combinatorics.org/files/v26i4p45/08x11_5Qalmost.html}{8}
placements is in Figure 1.

\begin{figure} \label{fig8x11almost}
\begin{center}
\begin{picture}(120,170)
\put(-70,-40){\resizebox{90mm}{!}{\rotatebox{0}{\includegraphics{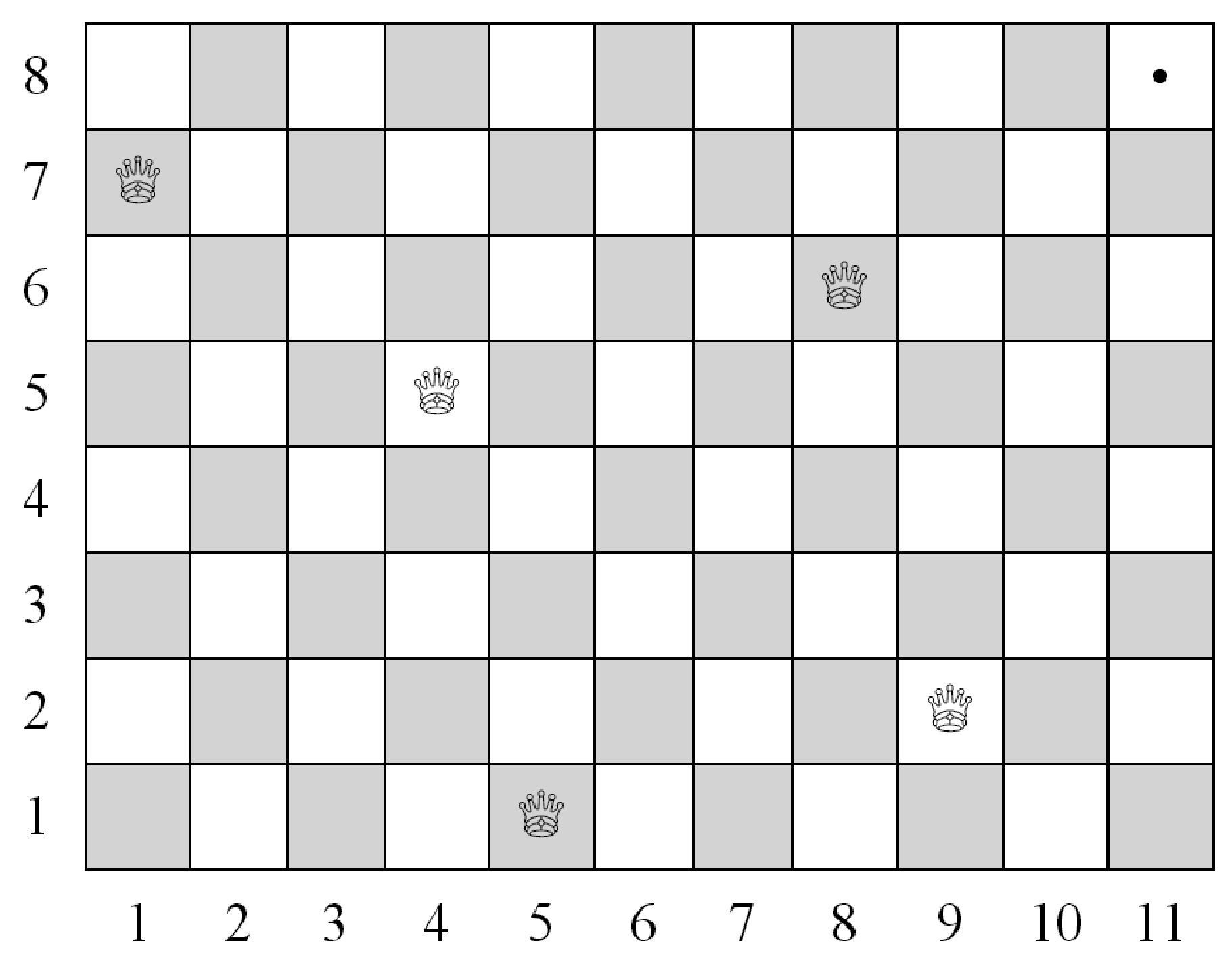}}}}
\end{picture}\\[12mm]
\end{center}
\caption{Five queens dominate $Q_{8 \times 11}$ except for one square (\hspace*{-0.2mm}$\bullet$)}
\end{figure}

We extend Cockayne's question to the rectangular case.

\begin{ques} \label{cockayne}
Column-wise monotonicity:  Does $\gamqq \leq \gamma(\q{m}{(n+1)})$ hold for $m \leq n$?

Row-wise monotonicity: Does $\gamqq \leq \gamma(\q{(m+1)}{n})$ hold for $m \leq n, (m, n) \neq (8, 11)$?
\end{ques}

We discuss one type of internal symmetry of \mds s that frequently occurs.
A \emph{foursome} is a set of four squares $(x+a, y+b)$, $(x-a, y-b)$, $(x-b, y+a)$, $(x+b, y-a)$,
where either each of $x, y, a, b$ is an integer or each is half an odd integer, and $a$ and $b$ are unequal and nonzero.
The  \emph{center}  of the foursome is the point $(x, y)$, which need not be a square
center.  For examples, see Figure 1 above, the f{\kern0pt}irst \mds s given for $Q_{9 \times 9}$ and $Q_{11 \times
11}$, as well as the f{\kern0pt}irst four \mds s given for $Q_{11 \times 12}$.

If a foursome $F$ is f{\kern0pt}lipped across any of the four lines through its center, the result is another
foursome $F'$ that occupies the same lines as $F$; this is illustrated in Figure 2.

\begin{figure}[h] \label{fig4some}
\unitlength 1mm
\begin{center}
\begin{picture}(80,70)
\put(3,-10){\resizebox{70mm}{!}{\rotatebox{0}{\includegraphics{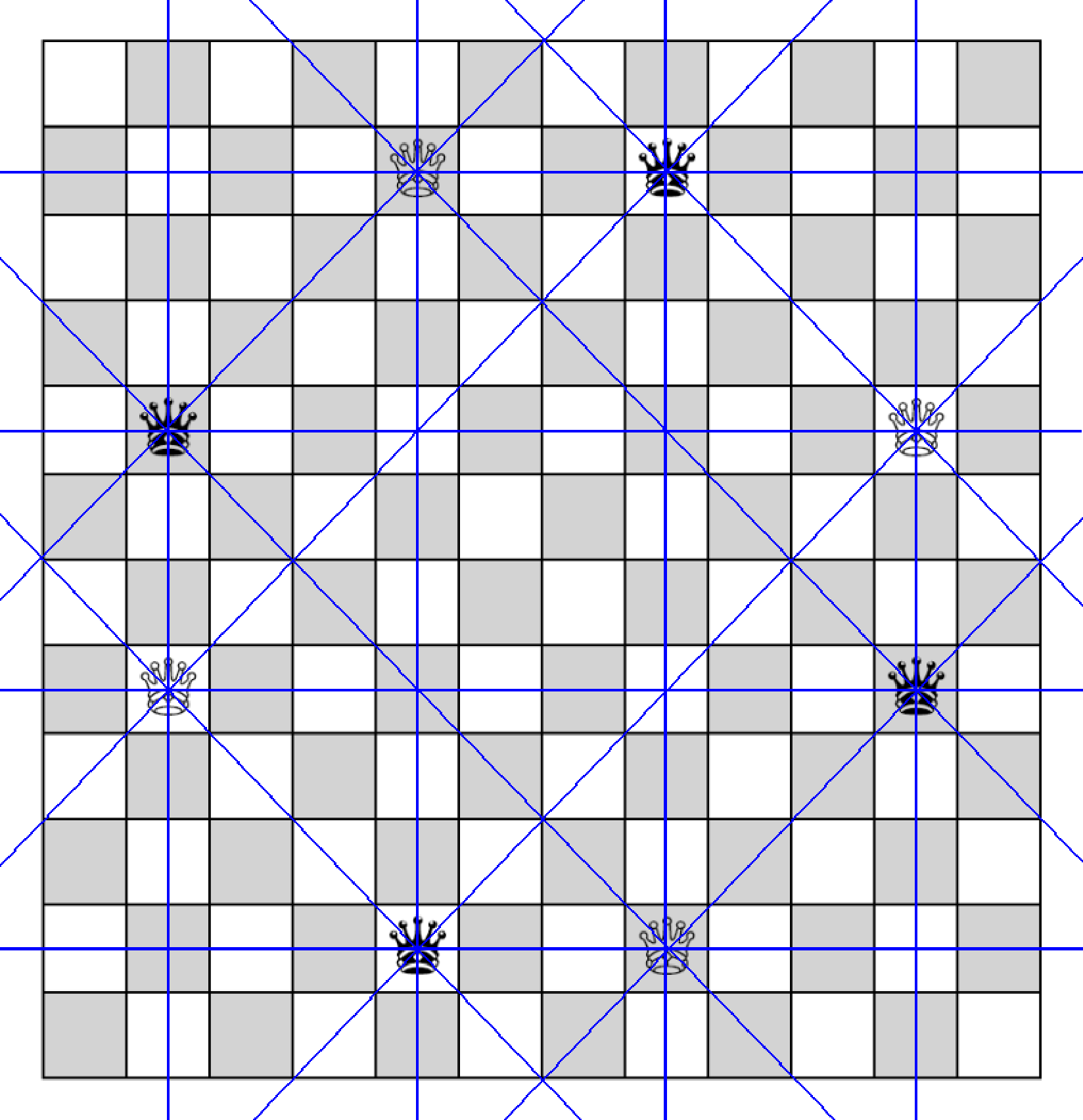}}}}
\end{picture}\\[12mm]
\end{center}
\setlength{\unitlength}{0.6mm}
\caption{The four squares with white queens are a foursome; the lines through its squares are shown.
Ref{\kern0pt}lecting this foursome across any of the four lines through its
center gives another foursome (the squares with black queens) that
occupies the same lines.}

\end{figure}

Thus if a dominating set $D$ of $Q_{m \times n}$ contains $F$, we may replace $F$ in $D$ with $F'$ and
obtain a dominating set $D'$ of the same size as $D$.  Usually $D$ and $D'$ are not equivalent.

As an example, we analyze the \mds s of $Q_{11 \times 17}$, which have size 8.
Up to equivalence there are 131 solutions, shown in the
\href{https://www.combinatorics.org/files/v26i4p45/11x17_8Q.html}{f{\kern0pt}ile}.
Of these,  85 have no foursomes, 41 have exactly one foursome, four
(\#\href{https://www.combinatorics.org/files/v26i4p45/11x17_8Q.html\#Solution125}{125},
\href{https://www.combinatorics.org/files/v26i4p45/11x17_8Q.html\#Solution69}{69},
\href{https://www.combinatorics.org/files/v26i4p45/11x17_8Q.html\#Solution70}{70},
\href{https://www.combinatorics.org/files/v26i4p45/11x17_8Q.html\#Solution62}{62})
have  exactly   two foursomes, and one
(\#\href{https://www.combinatorics.org/files/v26i4p45/11x17_8Q.html\#Solution76}{76})
has 3 foursomes.

We may def{\kern0pt}ine a relation on the set ${\cal S}(11, 17)$ of minimum
dominating sets of \q{11}{17} by saying that two sets are related if either they are equivalent, or
f{\kern0pt}lipping a foursome of the f{\kern0pt}irst set yields a
set equivalent to the second set. This relation is ref{\kern0pt}lexive and symmetric, and its transitive closure gives a
partition ${\cal P}(11, 17)$ of ${\cal S}(11, 17)$, which may also be regarded
as a partition of the set of (isometric) equivalence classes, as we will do.

For example, solution
\#\href{https://www.combinatorics.org/files/v26i4p45/11x17_8Q.html\#Solution125}{125}
has two foursomes: one centered at $(12,6)$ with $(a, b) = (4, 2)$,  and one centered at $(9, 7)$
with $(a, b) = (3, -1)$.  Flipping the f{\kern0pt}irst gives solution
\#\href{https://www.combinatorics.org/files/v26i4p45/11x17_8Q.html\#Solution124}{124}.
If instead we
f{\kern0pt}lip the second, we get the ref{\kern0pt}lection of
\#\href{https://www.combinatorics.org/files/v26i4p45/11x17_8Q.html\#Solution125}{125}
across its vertical line of symmetry.
This implies that one part of the partition ${\cal P}(11, 17)$ contains just the equivalence classes of
\href{https://www.combinatorics.org/files/v26i4p45/11x17_8Q.html\#Solution125}{125}
and
\href{https://www.combinatorics.org/files/v26i4p45/11x17_8Q.html\#Solution124}{124},
and we denote this part by
\{\href{https://www.combinatorics.org/files/v26i4p45/11x17_8Q.html\#Solution125}{125},
\href{https://www.combinatorics.org/files/v26i4p45/11x17_8Q.html\#Solution124}{124}\}.
It is then straightforward to see that ${\cal P}(11, 17)$ has 85 parts with one member and  20  parts with two members:
\{\href{https://www.combinatorics.org/files/v26i4p45/11x17_8Q.html\#Solution8}{8},
\href{https://www.combinatorics.org/files/v26i4p45/11x17_8Q.html\#Solution9}{9}\},
\{\href{https://www.combinatorics.org/files/v26i4p45/11x17_8Q.html\#Solution20}{20},
\href{https://www.combinatorics.org/files/v26i4p45/11x17_8Q.html\#Solution24}{24}\}
\{\href{https://www.combinatorics.org/files/v26i4p45/11x17_8Q.html\#Solution19}{19},
\href{https://www.combinatorics.org/files/v26i4p45/11x17_8Q.html\#Solution23}{23}\},
\{\href{https://www.combinatorics.org/files/v26i4p45/11x17_8Q.html\#Solution21}{21},
\href{https://www.combinatorics.org/files/v26i4p45/11x17_8Q.html\#Solution25}{25}\},
\{\href{https://www.combinatorics.org/files/v26i4p45/11x17_8Q.html\#Solution22}{22},
\href{https://www.combinatorics.org/files/v26i4p45/11x17_8Q.html\#Solution26}{26}\},
\{\href{https://www.combinatorics.org/files/v26i4p45/11x17_8Q.html\#Solution15}{15},
\href{https://www.combinatorics.org/files/v26i4p45/11x17_8Q.html\#Solution27}{27}\},
\{\href{https://www.combinatorics.org/files/v26i4p45/11x17_8Q.html\#Solution13}{13},
\href{https://www.combinatorics.org/files/v26i4p45/11x17_8Q.html\#Solution6}{6}\},
\{\href{https://www.combinatorics.org/files/v26i4p45/11x17_8Q.html\#Solution12}{12},
\href{https://www.combinatorics.org/files/v26i4p45/11x17_8Q.html\#Solution5}{5}\},
\{\href{https://www.combinatorics.org/files/v26i4p45/11x17_8Q.html\#Solution128}{128},
\href{https://www.combinatorics.org/files/v26i4p45/11x17_8Q.html\#Solution129}{129}\},
\{\href{https://www.combinatorics.org/files/v26i4p45/11x17_8Q.html\#Solution73}{73},
\href{https://www.combinatorics.org/files/v26i4p45/11x17_8Q.html\#Solution68}{68}\},
\{\href{https://www.combinatorics.org/files/v26i4p45/11x17_8Q.html\#Solution127}{127},
\href{https://www.combinatorics.org/files/v26i4p45/11x17_8Q.html\#Solution130}{130}\},
\{\href{https://www.combinatorics.org/files/v26i4p45/11x17_8Q.html\#Solution126}{126},
\href{https://www.combinatorics.org/files/v26i4p45/11x17_8Q.html\#Solution131}{131}\},
\{\href{https://www.combinatorics.org/files/v26i4p45/11x17_8Q.html\#Solution103}{103},
\href{https://www.combinatorics.org/files/v26i4p45/11x17_8Q.html\#Solution97}{97}\},
\{\href{https://www.combinatorics.org/files/v26i4p45/11x17_8Q.html\#Solution125}{125},
\href{https://www.combinatorics.org/files/v26i4p45/11x17_8Q.html\#Solution124}{124}\},
\{\href{https://www.combinatorics.org/files/v26i4p45/11x17_8Q.html\#Solution39}{39},
\href{https://www.combinatorics.org/files/v26i4p45/11x17_8Q.html\#Solution44}{44}\},
\{\href{https://www.combinatorics.org/files/v26i4p45/11x17_8Q.html\#Solution96}{96},
\href{https://www.combinatorics.org/files/v26i4p45/11x17_8Q.html\#Solution95}{95}\},
\{\href{https://www.combinatorics.org/files/v26i4p45/11x17_8Q.html\#Solution63}{63},
\href{https://www.combinatorics.org/files/v26i4p45/11x17_8Q.html\#Solution62}{62}\},
\{\href{https://www.combinatorics.org/files/v26i4p45/11x17_8Q.html\#Solution72}{72},
\href{https://www.combinatorics.org/files/v26i4p45/11x17_8Q.html\#Solution70}{70}\},
\{\href{https://www.combinatorics.org/files/v26i4p45/11x17_8Q.html\#Solution101}{101},
\href{https://www.combinatorics.org/files/v26i4p45/11x17_8Q.html\#Solution77}{77}\},
\{\href{https://www.combinatorics.org/files/v26i4p45/11x17_8Q.html\#Solution100}{100},
\href{https://www.combinatorics.org/files/v26i4p45/11x17_8Q.html\#Solution78}{78}\}.
There are also two parts with three members:
\{\href{https://www.combinatorics.org/files/v26i4p45/11x17_8Q.html\#Solution80}{80},
\href{https://www.combinatorics.org/files/v26i4p45/11x17_8Q.html\#Solution71}{71},
\href{https://www.combinatorics.org/files/v26i4p45/11x17_8Q.html\#Solution69}{69}\} and
\{\href{https://www.combinatorics.org/files/v26i4p45/11x17_8Q.html\#Solution79}{79},
\href{https://www.combinatorics.org/files/v26i4p45/11x17_8Q.html\#Solution76}{76},
\href{https://www.combinatorics.org/files/v26i4p45/11x17_8Q.html\#Solution75}{75}\}.

It would be possible to reduce the size of the
\href{https://www.combinatorics.org/ojs/index.php/eljc/article/view/v26i4p45/HTML}{appendix}
by giving for each $(m, n)$ one solution
from each part of the partition ${\cal P}(m, n)$ rather than one
solution from each isometry equivalence class.
But when two solutions dif{\kern0pt}fer by the f{\kern0pt}lip of a foursome, it is not clear
which is most useful to see, so we have not done this.

\section{Lower bounds on queen domination numbers}

We begin by looking at what happens when the board is far from square.

\begin{prop} \label{prop1}
If $n \geq 3m-2$, then $ \gamma(\q{m}{n}) = m.$
\end{prop}

\begin{proof}
Each queen attacks all squares in her own row, but at most
three squares in any other row. Thus $m-1$ queens occupy at most
$m-1$ rows and cover at most $3(m-1)$ squares in any row that
does not contain a queen. On the other hand, $m$ queens are
certainly suf{\kern0pt}f{\kern0pt}icient.
\end{proof}

Note that $\gamma(\q{3}{6}) = 2$ (see the set given immediately after Theorem \ref{by3k}) and
$\gamma(\q{5}{12}) = 4$ (see the database), but as shown by our computations,
for $m = 4, 6, 7$, \gamqq\ reaches $m$ before $n$ reaches $3m-2$.

We change viewpoint slightly, focusing on the size of the dominating
set rather than the dimensions of the board.  For each positive integer $m$, let
$g(m)$ be the largest integer such that $m$ queens can cover the $(m+1) \times g(m)$
board.  (Proposition \ref{prop1} asserts $g(m) \leq 3m$.)
Let $g^{*}(m)$ be the largest integer such that $m$ queens, no two in a
row, can cover the $(m+1) \times g^{*}(m)$ board.

\begin{thm} \label{by3k}
If $g(m) = 3m$ then  $m \equiv 0, 1, 2, 3, \mbox{or } 4 \mbox{  (mod $9$)}$.

Take the board to be a rectangle in the Cartesian plane with sides parallel to
the axes, board center at the origin, and squares of edge length two.

Assume that $D$ is a dominating set of size $m$ for the $(m+1) \times 3m$ board.  One of two
cases occurs:

(1) there is only one empty row, which without loss of generality has index $h$, $0 \leq h \leq
m$, and each other row contains exactly one square of $D$, or

(2) there are exactly two empty rows; in this case $m$ is even, row 0 contains two squares of $D$, the empty
rows are indexed $\pm h$ for some $h$, $0 < h \leq m$, and each other row contains exactly one square of $D$.

In either case:\\
if $m \equiv 0 \mbox{ or } 4 \mbox{ (mod $9$)}$, then $h \equiv 0 \mbox{(mod $3$)}$; \\
if $m \equiv 1 \mbox{ or } 3 \mbox{ (mod $9$)}$, then $h \not \equiv 0 \mbox{ (mod $3$)}$.
\end{thm}

\begin{proof}
The $3m$ column indices are $x = 1-3m, 3-3m, \ldots, 3m-3, 3m-1$
and the $m+1$ row indices are $y = -m, 2-m, \ldots, m-2, m$.

Let $D = \{ (x_{i}, y_{i})\mbox{ : }1 \leq i \leq m\}$ be a
dominating set of the $(m+1) \times 3m$-board. Since $|D| = m$ and there are $m+1$ rows, at least one row contains
no square of $D$.  Let $h$ be the index of such a row.  Since $|D| = m$, and a queen covers at most three squares of any
line she does not occupy, the $3m$ squares of row $h$
are each covered exactly once by $D$.  For each $(x, y) \in D$, the squares of
row $h$ covered by $(x, y)$ are $(x-(y-h), h), (x, h), (x+(y-h), h)$. We can add
up the squares of the column indices of the squares in row $h$ in two ways,
giving the equation
\[
(1-3m)^{2} + (3-3m)^{2} + \ldots +
(3m-1)^{2} =
\sum_{i=1}^{m} \left[ (x_{i}-(y_{i}-h))^{2} + x_{i}^{2} + (x_{i}+(y_{i}-h))^{2}
\right].
\]
This reduces to
\begin{equation} \label{longeq2}
2 \binom{3m+1}{3} = 3 \sum_{i=1}^{m} x_{i}^{2}
+ 2 \sum_{i=1}^{m} y_{i}^{2} - 4h \sum_{i=1}^{m} y_{i} + 2mh^{2}.
\end{equation}
For a particular dominating set $D$ we may regard this as a quadratic equation in $h$,
so there are at most two empty rows.

If there are two empty rows $h_{1}, h_{2}$, then (\ref{longeq2}) implies
$2mh_{1}^{2} - 4h_{1}\sum_{i=1}^{m} y_{i} = 2mh_{2}^{2} - 4h_{2}\sum_{i=1}^{m} y_{i}$ and then
\begin{equation} \label{h1ph2}
  \sum_{i=1}^{m} y_{i} = \frac{m(h_{1}+h_{2})}{2}.
\end{equation}
As $|D| = m$, there is exactly one row, say $l$, with two queens, and all rows
except $h_{1}, h_{2}, l$ have just one queen. Thus $\sum_{i=1}^{m} y_{i} =
-h_{1}-h_{2}+l$. With (\ref{h1ph2}) this implies $-h_{1}-h_{2}+l = m(h_{1}+h_{2})/2$
and then $l = (m+2)(h_{1}+h_{2})/2$.  From $-m \leq l \leq m$ we have $-2m \leq (m+2)(h_{1}+h_{2}) \leq 2m$
so $-2 < h_{1}+h_{2} < 2$. But $h_{1}+h_{2}$
is even since all row indices have the same parity. Thus $h_{1}+h_{2} = 0$, so $l = 0$ and all row indices
are even, which implies $m$ is even.  So there is $h$, $0 < h \leq m$, such that the empty rows are $\pm h$.
Here $\sum_{i=1}^{m} y_{i} = 0$ and $\sum_{i=1}^{m} y_{i}^{2} = 2 \binom{m+2}{3} -
2h^{2}$, and (\ref{longeq2}) becomes
\begin{equation} \label{2empty}
  2 \binom{3m+1}{3} - 4 \binom{m+2}{3} = 3 \sum_{i=1}^{m} x_{i}^{2} + 2(m-2)h^{2}.
\end{equation}
If instead there is only one empty row $h$, then we may assume $0 \leq h \leq m$ by f{\kern0pt}lipping across
the $x$-axis if necessary.  Then  $\sum_{i=1}^{m} y_{i} = -h$ and
$\sum_{i=1}^{m} y_{i}^{2} = 2 \binom{m+2}{3} - h^{2}$, and (\ref{longeq2}) gives
\begin{equation} \label{1empty}
  2 \binom{3m+1}{3} - 4 \binom{m+2}{3} = 3 \sum_{i=1}^{m} x_{i}^{2} + 2(m+1)h^{2}.
\end{equation}

The left sides of (\ref{2empty}) and (\ref{1empty}) reduce to
$m(25m^{2}-6m-7)/3$.  Multiplying either of (\ref{2empty}) and (\ref{1empty}) by 3
and reducing modulo 9 gives the congruence
\begin{equation} \label{maincong}
  m(25m^{2}-6m-7) \equiv -3(m-2)h^{2} {\mbox{ (mod 9)}}.
\end{equation}
For $m \equiv {5, 6, 7} \mbox{ or } {8} \mbox{  (mod 9)}$, (\ref{maincong}) leads
to $h^{2} \equiv -1 \mbox{ (mod 3)}$ or another impossibility.  For
$m \equiv 0 \mbox{ or } 4 \mbox{  (mod 9)}$, (\ref{maincong}) implies
$h^{2} \equiv 0 \mbox{ (mod 3)}$ and thus $h \equiv 0 \mbox{ (mod 3)}$.
For $m \equiv 1 \mbox{ or } 3 \mbox{  (mod 9)}$, (\ref{maincong}) gives
$h^{2} \equiv 1 \mbox{ (mod 3)}$, so $h \equiv \pm 1 \mbox{ (mod 3)}$.
(For $m \equiv 2 \mbox{  (mod 9)}$, (\ref{maincong}) is satisf{\kern0pt}ied for any $h$.)
\end{proof}

Only one example of the second case of Theorem \ref{by3k} is known: for $m=2$,
the set $D = \{ (-3, 0), (3, 0)\}$ covers the $3 \times 6$ board.  A computer
search shows that there is no other example with $m \leq 40$. It seems likely
that no other example exists; if  for some $m$ there is such a set
$D = \{ (x_{i}, y_{i}) \mbox{ : }1 \leq i \leq m\}$, we have been able to show
that
$0 = \sum x_{i} = \sum x_{i}y_{i} =  \sum x_{i}^{2}y_{i} = \sum (x_{i}^{3}y_{i}+2x_{i}y_{i}^{2})$.

By computer search (see the \href{https://www.combinatorics.org/files/v26i4p45/appendix-Theorem2.html}{f{\kern0pt}ile}) we have shown that for all $m$ such that $m \leq 40$, $m \neq 3$, and
$m \equiv 0, 1, 2, 3, \mbox{or } 4 \mbox{  (mod 9)}$,
there exist sets of $m$ queen squares dominating $Q_{(m+1)\times 3m}$, with $h$ taking all
values not ruled out by Theorem \ref{by3k}.   We believe $g(m) = g^{*}(m) = 3m$ holds for all $m \neq 3$
with $m \equiv 0, 1, 2, 3, \mbox{or } 4 \mbox{  (mod 9)}$.

The following is immediate from Theorem \ref{by3k}.

\begin{cor}
  If $m \equiv 5, 6, 7, \mbox{or } 8  \mbox{  (mod $9$)}$ then $\gamma(Q_{(m+1) \times 3m}) = m+1$.
\end{cor}

We now examine the cases $m \equiv 5, 6, 7, 8 \mbox{  (mod $9$)}$.

\begin{thm} \label{upperbounds}
  If $m \equiv 5, 6, \mbox{or } 7 \mbox{ (mod $9$)}$ then $g^{*}(m) \leq 3m-2$.\\
  If $m \equiv 8 \mbox{ (mod $9$)}$ then $g^{*}(m) \leq 3m-4$.

  Take the board to be a rectangle in the Cartesian plane with sides parallel to
  the axes, board center at the origin, and squares of edge length two.

  For  $m \equiv 6 \mbox{ or } 7 \mbox{ (mod $9$)}$,  assume that $D$ is a set of $m$ squares that dominates
  the $(m+1) \times (3m-2)$ board, occupying all but row $h$.  If $m \equiv 6 \mbox{ (mod $9$)}$ then
  $h \not \equiv 0 \mbox{ (mod $3$)}$.
  If $m \equiv 7 \mbox{ (mod $9$)}$ then $h \equiv 0 \mbox{ (mod $3$)}$.
\end{thm}

\begin{proof}

Let $m$ be a positive integer and $j$ an integer with $0 \leq j \leq 4$. Let $D$ be a dominating set of $m$ queens
for the $(m+1) \times (3m-j)$-board, with only row $h$ empty.

The $3m-j$ column indices form the set $S_{\rm col} = \{ j+1-3m, j+3-3m, \ldots, 3m-j-3, 3m-j-1\}$ and the $m+1$ row
indices form the set $S_{\rm row} = \{-m, 2-m, \ldots, m-2, m \}$.

It will be useful to consider row $h$ extended beyond the board, and to look at the congruence
classes modulo 3 of the column indices.  To this end, for each integer $i$, let
$C'_{i} = \{x \in \zah \mbox{ : } x \equiv i \mbox{ (mod 3)} \}$.  The
restriction of $C'_{i}$ to column indices of the $(m+1) \times (3m-j)$-board is $C_{i} = C'_{i} \cap
S_{\rm col}$. We write $c_{i}$ for the size of $C_{i}$.  By symmetry, $C_{-1} = \{-x \mbox{ : } x \in C_{1}\} = -C_{1}$
so $c_{-1} = c_{1}$.

For each integer $i$ let $R_{i} = \{y \in S_{\rm row} \mbox{ : } y \equiv i \mbox{ (mod 3)}
\}$.  Write $r_{i}$ for the size of $R_{i}$.  By symmetry, $R_{-1} = - R_{1}$ so $r_{-1} = r_{1}$.

\begin{lemma} \label{techlemma}
  Choose $s, b \in \{ -1, 0, 1 \}$ such that $s \equiv j$ and $b \equiv k+1 \mbox{ (mod 3)}$.
  Then $c_{-1} = c_{1} = c_{0} + s$ and  $r_{-1} = r_{1} = r_{0} - b$.

\end{lemma}

\begin{proof} The facts that $c_{-1} = c_{1}$, $c_{-1} + c_{0} + c_{1} = 3m-j$ and the $c_{i}$'s dif{\kern0pt}fer by at most one
  imply the f{\kern0pt}irst equation, and the second is similar.
\end{proof}

Let $p$ be the number of squares $(x, y)$ of $D$ with $y \not \equiv h \mbox{ (mod 3)}$.  For each $i \in \{ -1, 0, 1 \}$, let
$t_{i}$ be the number of squares $(x, y)$ of $D$ such that $x \equiv i \mbox{ (mod 3)}$ and $y \equiv h \mbox{ (mod 3)}$.
Let $t = t_{-1} + t_{0} + t_{1}$. Thus $|D| = p + t_{-1} + t_{0} + t_{1} = p + t$.

Let $(x, y)$ be a square of $D$.
The squares in the extension of row $h$ covered by $(x, y)$ are $(x-(y-h), h), (x, h), (x+(y-h), h)$,
some of which may be of{\kern0pt}f the $(m+1) \times (3m-j)$-board.
Their $x$-coordinates $x - (y-h), x, x+(y-h)$ are an arithmetic progression with
dif{\kern0pt}ference $y-h$.  If $y \equiv h \mbox{ (mod 3)}$ then all of $x - (y-h), x, x+(y-h)$
are in $C'_{x}$.
If $y \not \equiv h \mbox{ (mod 3)}$ then the three values $x - (y-h), x, x+(y-h)$ are
dif{\kern0pt}ferent modulo 3, so they contribute one member to each of $C'_{-1}, C'_{0},
C'_{1}$.  Thus for each $i \in \{ -1, 0, 1\}$, $p + 3t_{i}$
is the number of covers (with multiplicity) of squares $(x, h)$ with $x \in
C'_{i}$.

For each $i \in \{ -1, 0, 1 \}$ let $a_{i}$ be the number of ``wasted covers''
by $D$ of squares $(x, h)$ with $x \in C'_{i}$.  That is, $a_{i}$ counts every cover of any square $(x, h)$
that is of{\kern0pt}f the board ($|x| > 3m-j-1$) and all but one cover of each multiply
covered square $(x,h)$ on the board.  Thus each $a_{i}$ is nonnegative, and since each of the $m$
squares of $D$ covers 3 squares of the extended row $h$,
$a_{-1} + a_{0} + a_{1} = j$.

Since each square in row $h$ of the board is covered by $D$, we get a system of equations:
\begin{eqnarray} 
  p + 3t_{-1} & = & c_{-1} + a_{-1} \label{one} \\
  p + 3t_{0} & = & c_{0} + a_{0}  \label{two} \\
  p + 3t_{1} & = & c_{1} + a_{1} \label{three} \\
  a_{-1} + a_{0} + a_{1} & = & j \label{four}.
\end{eqnarray}

As $c_{-1} = c_{1}$,
subtracting (\ref{one}) from (\ref{three}) shows
\begin{equation} \label{cong1}
  a_{-1} \equiv a_{1} \mbox{ (mod 3)}.
\end{equation}


For a dominating set $D$ as hypothesized to exist, it is necessary that
the total number $t$ of squares $(x,y)$ in $D$ with $y \equiv h \mbox{ (mod 3)}$
is one less than the number $r_{h}$ of rows in $R_{h}$.
When $t = r_{h} - 1$, we will say $h$ is \emph{eligible} for $t$.
\vspace{.1in}



Theorem \ref{by3k} covers the case $j = 0$, so
we pass to less wide boards, only considering $m \equiv 5, 6, 7, 8 \beem{9}$.

\noindent Let $j=1$. From (\ref{cong1}) we have $a_{-1}=a_{1}=0$ and then $a_{0}=1$.
Here the $s$ of Lemma \ref{techlemma} is 1 so $c_{-1} = c_{1} = m, c_{0} = m-1$.
Then equations (\ref{one}-\ref{three}) imply $t_{-1} = t_{0} = t_{1}$, so $t \equiv 0 \mbox{ (mod
3)}$, and the following analysis shows that no $h$ is eligible for any $t$ for $m \equiv 5, 6, 7, 8 \beem{9}$.

\indent If $m \equiv 5 \beem{9}$ then $r_{-1} = r_{1} = r_{0} = (m+1)/3 \equiv -1 \beem{3}$.\\
\indent If $m \equiv 6 \beem{9}$ then $r_{-1} = r_{1} = m/3 \equiv -1 \beem{3}, r_{0} = (m/3)+1 \equiv 0 \beem{3}$.\\
\indent If $m \equiv 7 \beem{9}$ then $r_{-1} = r_{1} = (m+2)/3 \equiv 0 \beem{3}, r_{0} = (m-1)/3 \equiv -1 \beem{3}$.\\
\indent If $m \equiv 8 \beem{9}$ then $r_{-1} = r_{1} = r_{0} = (m+1)/3 \equiv 0 \beem{3}$.

\noindent Let $j=2$. Since (\ref{four}) here implies all $a_{i} \leq 2$, (\ref{cong1}) gives
$a_{-1} = a_{1}$, and then (\ref{one}-\ref{three}) imply $t_{-1} = t_{1}$.  The $s$
of Lemma \ref{techlemma} is $-1$ so $c_{-1} = c_{1} = m, c_{0} = m+1$. There are
two possibilities:

$(a_{-1}, a_{0}, a_{1}) = (0, 2, 0)$, when equations (\ref{one}-\ref{three})
imply $t_{0} = t_{1}+1$, so $t \equiv 1 \beem{3}$, or

$(a_{-1}, a_{0}, a_{1}) = (1, 0, 1)$, when equations (\ref{one}-\ref{three})
imply $t_{0} = t_{1}$, so $t \equiv 0 \beem{3}$.

Then examining the values of the $r_{i}$'s found above for $m \equiv 5, 6, 7, 8
\beem{9}$, we see:\\ For $m \equiv 5 \beem{9}$, all $h$ are eligible for $t \equiv 1 \beem{3}$ and none for
$t \equiv 0 \beem{3}$;\\
For $m \equiv 6 \beem{9}$, $h \not \equiv 0 \beem{3}$ are eligible for $t \equiv 1 \beem{3}$ and none for
$t \equiv 0 \beem{3}$;\\
For $m \equiv 7 \beem{9}$, $h \equiv 0 \beem{3}$ is eligible for $t \equiv 1 \beem{3}$ and none for
$t \equiv 0 \beem{3}$;\\
For $m \equiv 8 \beem{9}$, no $h$ is eligible for either $t$.

We continue with $m \equiv 8 \beem{9}$.

\noindent Let $j=3$.  Here $c_{-1} = c_{0} = c_{1} = m-1$ so from equations (\ref{one}-\ref{three})
we see all $a_{i}$'s are congruent modulo 3. Either $(a_{-1}, a_{0}, a_{1}) = (1, 1, 1)$,
and then $t_{-1} = t_{0}= t_{1}$, so $t \equiv 0 \beem{3}$, or one of the
$a_{i}$'s is 3 and the other two are zero, which gives $t \equiv 1 \beem{3}$.
For $m \equiv 8 \beem{9}$, neither of these gives an eligible $h$, as before.

\noindent Let $j=4$. Here there are more possibilities for $(a_{-1}, a_{0},
a_{1})$, but the only helpful one for $m \equiv 8 \beem{9}$ is $(2, 0, 2)$, which
gives $t \equiv -1 \beem{3}$, with all $h$ eligible.
\end{proof}

Computer search reveals (see the
\href{https://www.combinatorics.org/files/v26i4p45/appendix-Theorem4.html}{f{\kern0pt}ile})
that for $m \leq 40$ and $ m \equiv 5,6,7 \beem{9},$
all minimum dominating sets of \q{m+1}{3m-2}  have just one empty row,
and all eligible values of $h$ actually occur. For $ m \equiv 8 \beem{9},$
Theorem 4  does not say any $h$ are ineligible; indeed, our search has found solutions
with one empty row for all the $h$ values.
The only board size in the $3 \leq m \leq 40$ range where the minimum dominating sets found have
two empty rows is \q{9}{20}. Those dominating sets demonstrate numerous
patterns of pairs of empty rows, as shown in this
\href{https://www.combinatorics.org/files/v26i4p45/appendix-9x20Spec.html}{f{\kern0pt}ile}.\\


We next show that by ``pasting together'' dominating sets of a certain type, we can
extend the range of values for which the bounds of Theorem \ref{upperbounds} are
known to be exact.

Say that a \emph{topless} dominating set for the $(m+1)\times n$
board is a dominating set of size $m$  having one square in each row
except the top row, which is empty.

\begin{prop} \label{pasting}
  (A) Suppose that for $i = 1, 2$ there is a topless dominating set
  of $m_{i}$
  queens for the $(m_{i}+1) \times n_{i}$ board. Then there is a dominating set
  of $m_{1} + m_{2}$ queens for the $(m_{1}+m_{2}+1) \times (n_{1}+n_{2})$
  board.

  (B) Let $k$ be a positive integer. Suppose that for each $l$, $1 \leq l \leq k$, and for $i = 0, 1, 2$, there exist topless
  dominating sets of size $9l+i$ for the $(9l+i+1) \times (27l + 3i)$ board.
  Then for each $m \neq 3$, $1 \leq m \leq 9k+8$, there is a dominating set of
  size $m$ for the $(m+1) \times (3m-j)$ board, where $j=0$ if
  $m \equiv 0, 1, 2, 3, \mbox{or } 4 \mbox{  (mod $9$)}$,
  $j=2$ if $m \equiv 5, 6, \mbox{or } 7 \mbox{  (mod $9$)}$, and $j=4$ if
  $m \equiv 8 \mbox{  (mod $9$)}$.

\end{prop}

\begin{proof}  We return to our usual scheme of indexing columns and
  rows from the bottom left board corner.
  For (A): For $i = 1, 2$, let $S_{i}$ be a topless dominating set of $m_{i}$
  queens for the $(m_{i}+1) \times n_{i}$ board.
  On the $(m_{1}+m_{2}+1) \times (n_{1}+n_{2})$ board, the squares in which the columns
  indexed $1$ to $n_{1}$ and the rows indexed $1$ to $m_{1}+1$ meet
 form a copy of the $(m_{1}+1)\times n_{1}$ board.  Place a copy $S_{1}'$ of $S_{1}$
  on that.  The squares in which the columns
  indexed $n_{1}+1$ to $n_{1}+n_{2}$ and the rows indexed $m_{1}+1$ to $m_{1}+m_{2}+1$ meet
  form a copy of the $(m_{2}+1)\times n_{2}$ board.  Place a copy $S_{2}'$ of $S_{2}$, rotated
  by a half turn, on that.  The union of the two copies is a set $S_{1} \oplus
  S_{2}$ of  $m_{1}+m_{2}$ squares that leaves only row $m_{2}+1$ empty on the $(m_{1}+m_{2}+1) \times (n_{1}+n_{2})$
  board.  As $S_{1}'$ covers the left $n_{1}$ squares of that row and $S_{2}'$
  covers the remainder, $S_{1} \oplus S_{2}$ dominates the $(m_{1}+m_{2}+1) \times (n_{1}+n_{2})$
  board.

  For (B): We will use topless dominating sets of the $(m+1) \times (3m-j)$ board
  for $(m,j)$ =
  \href{https://www.combinatorics.org/files/v26i4p45/appendix-Theorem2.html#m1h1}{(1,0)},
  \href{https://www.combinatorics.org/files/v26i4p45/appendix-Theorem4.html#m5h5}{(5,2)},
  \href{https://www.combinatorics.org/files/v26i4p45/appendix-Theorem4.html#m8h8}{(8,4)}, \mbox{and} \hspace{.03in}
  \href{https://www.combinatorics.org/files/v26i4p45/appendix-Theorem2.html#m11h11}{(11,0)}.
  Assume the hypotheses; we then need to prove the existence of (minimum) dominating sets of size
  $m$ for the $(m+1) \times (3m-j)$ boards in the ranges claimed.
  For $m < 18$ the database (see the f{\kern0pt}iles
  \href{https://www.combinatorics.org/files/v26i4p45/appendix-Theorem2.html}{here}
  and
  \href{https://www.combinatorics.org/files/v26i4p45/appendix-Theorem4.html}{here})
   contain the claimed dominating sets. For other board sizes of $m = 9l, 9l + 1, 9l + 2, \, 2 \leq l \leq k$,
  we are assuming there are such sets (in fact, topless).
  For $m = 9l + 3$, we use part (A) to ``paste together" topless sets for $m = 9l + 2$ and $m = 1$.
  For $m = 9l + 4$, we paste together topless sets for $m = 9(l - 1) + 2$ and $m = 11$.
  Pasting together a topless set for each of $m = 9l, 9l + 1, 9l + 2$ with a topless set for $m = 5$ gives
  the result for $m = 9l + 5, 9l + 6, 9l + 7$.
  Finally, pasting together topless sets for $m = 9l$ and $m = 8$ gives the result for $m = 9l + 8$.
\end{proof}

 As mentioned after Theorem \ref{by3k}, computer search gave topless dominating sets as
 required in Proposition \ref{pasting}(B)
 for $k = 4$.
 It follows that $g^{*}(m)$ equals the bound of Theorem \ref{upperbounds}
 for $m \leq 44, m \neq 3$.

\vspace{.1in}

We next develop a lower bound for $\gamma(Q_{m \times n})$ for more general $m, n$.

Raghavan and Venketesan \cite{RV} and Spencer \cite{CO,WE} independently proved that
\begin{equation} \label{rvs}
\gamma(\q{n}{n}) \geq \left \lceil \frac{n-1}{2} \right \rceil.
\end{equation}

It has been shown \cite{FW} that $\gamma(\q{n}{n}) = (n-1)/2$ only for $n = 3, 11$.  Both of these values
are signif{\kern0pt}icant for our work here, as we now discuss.

A central queen on \q{3}{3}\ shows $\gamma(\q{3}{3}) = 1$. This simple fact has a useful generalization: if $C$ is
a central sub-board of \qq\ such that every square of \qq\ has a line meeting $C$, then a subset of $C$ that
occupies all those lines is a dominating set of \qq.  More than a hundred years ago, Szily \cite{V1, V2} gave
dominating sets of this type for \q{13}{13}\ and \q{17}{17}, which were later shown to be minimum. We found
that \q{13}{16} has a minimum dominating set (solution
\href{https://www.combinatorics.org/files/v26i4p45/13x16_8Q.html\#Solution23}{\#23}
in the database) of this \emph{centrally strong}  form
and have also used this idea to produce good upper bounds of \qq\ for some $m, n$, as shown below. \vsm

It follows from \cite{LB} that there are exactly two minimum
dominating sets for \q{11}{11}. Placing the origin of our
coordinate system at board center, these sets are ${\cal D} = \{ (0, 0)$,
$\pm(2, 4)$, $\pm(4, -2) \}$  (see Figure 3) and the ref{\kern0pt}lection of ${\cal D}$ across the
column $x = 0$.
\begin{figure}[h] \label{figq11}
\unitlength 1mm
\begin{center}
\begin{picture}(100,70)
\put(17,-5){\resizebox{70mm}{!}{\rotatebox{0}{\includegraphics{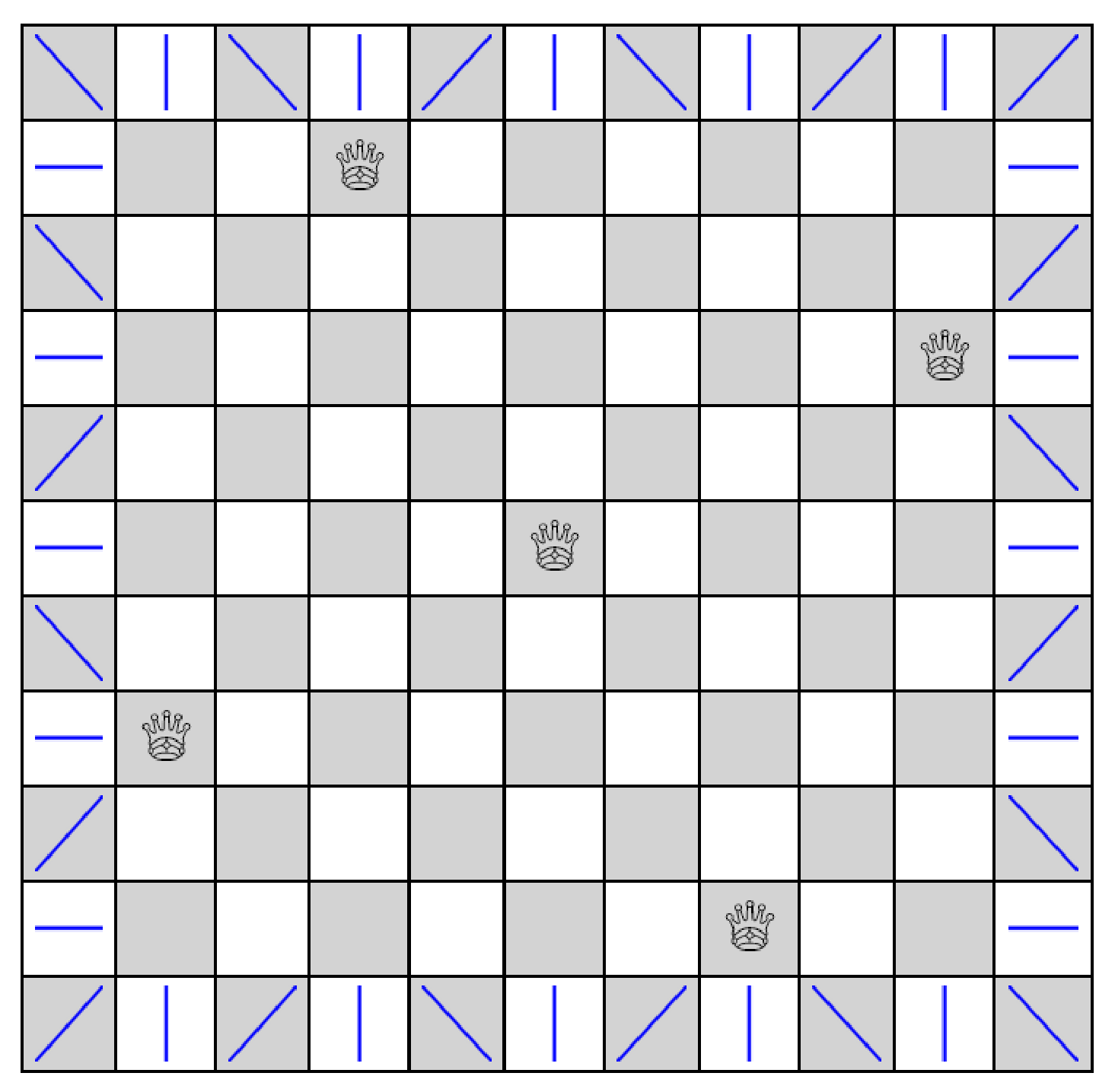}}}}
\end{picture}\\
\end{center}
\unitlength 0.6mm
\caption{The minimum dominating set ${\cal D}$ of $Q_{11 \times 11}$
is shown, with the unique cover of each edge square (see Corollary \ref{coro}) indicated.}
\end{figure}
So up to  equivalence  ${\cal D}$ is the unique minimum
dominating set of
\href{https://www.combinatorics.org/files/v26i4p45/11x11_5Q.html}{\q{11}{11}},
consisting of a foursome and a queen at  its center. This amazing
set has an inf{\kern0pt}luence on many other values of \gamqq.

First, since ${\cal D}$ f{\kern0pt}its on \q{9}{9}, by omitting edge rows and columns of \q{11}{11}\ we get dominating
sets of \qq\ for $(m, n) = (10, 11), (10, 10), (9, 11), (9, 10), (9, 9)$, and these turn out to be minimum dominating
sets. In a sense, the observed failure of monotonicity, $\gamma(\q{8}{11}) = 6 > 5 = \gamma(\q{9}{11})$,
occurs simply because ${\cal D}$ does not f{\kern0pt}it on \q{8}{11}.

Also, by adding edge rows or columns to \q{11}{11}\ and adding corner squares to ${\cal D}$, we obtain minimum
dominating sets for the values $\gamma(\q{11}{12}) = 6$, $\gamma(\q{11}{13}) = 7$,
$\gamma(\q{12}{12}) = 6$, $\gamma(\q{12}{13}) = 7$, $\gamma(\q{12}{14}) = 8$, $\gamma(\q{13}{13}) = 7$,
$\gamma(\q{13}{14}) = 8$, $\gamma(\q{14}{14}) = 8$, and $\gamma(\q{15}{15}) = 9$. Finally, it is shown
in \cite{UB} that ${\cal D}$ gives a set implying $\gamma(\q{53}{53}) = 27$.

It was observed by Eisenstein et al. \cite{EI} that if a dominating set $D$ of $\q{n}{n}$
contains no edge squares, the facts that there are $4(n-1)$ edge squares and every queen covers
eight edge squares imply $|D| \geq \lceil (n-1)/2 \rceil$.  This suggests the bound (\ref{rvs}).

A similar approach leads one to guess the bound of our
next theorem, but some care is needed to handle the general case.

\begin{thm} \label{main} Let $m, n$ be positive integers with $m \leq n$. Then

\begin{equation} \label{rvs2}
\gamma(Q_{m \times n}) \geq \min \left\{m, \left\lceil \frac{m+n-2}{4} \right\rceil\right\}.
\end{equation}
\end{thm}

\begin{proof} It suf{\kern0pt}f{\kern0pt}ices to show that if $\gamma(Q_{m \times n}) \leq m-1$ then
$\gamma(Q_{m \times n}) \geq (m+n-2)/4 $.  So we assume
that $\gamma(Q_{m \times n}) \leq m-1$.

First, suppose $\gamma(\q{m}{n})=m-1$. Then by Proposition \ref{prop1} we have $n <
3m-2$, which implies $m-1 > (m+n-2)/4$ as needed.

Thus we may take  $\gamma(\q{m}{n}) \leq m-2$
and let $D$ be a minimum dominating set of $Q_{m \times n}$.
Since $m \leq n$, there are
at least two rows and at least two columns that do not contain
squares of $D$. Let $a$ be the index of the leftmost empty
column, $b$ the index of the rightmost empty column, $c$ the
index of the lowest empty row, $d$ the index of the highest
empty row. The board has a rectangular sub-board $U$ with corner
squares $(a, c), (a, d), (b, c), \mbox{ and } (b, d)$. Let $E$ be
the set of edge squares of this sub-board.
We say that $U$ is the \emph{box} of $D$ and $E$ is the
\emph{box border} of $D$; these sets are def{\kern0pt}ined for any square set $D$ with
$|D| \leq m-2$.
Here $|E| = 2(d-c) +
2(b-a)$.

 Removing columns $a$ and $b$ and rows $c$ and $d$ divides  the board
into nine regions (some possibly empty). Let $C$ be the set of
squares of $D$ inside $U$; that is, $C = \{ (x, y) \in D \mbox{ : }
a<x<b \mbox{ and } c<y<d \}$. Let $T_{nw}$ be the set of squares
of $D$ in the ``northwest'' region of the $m \times n$ board; that
is, $T_{nw} = \{ (x, y) \in D \mbox{ : } x < a \mbox{ and } y > d
\}$. Similarly we label seven more subsets of $D$ by their
``geographic direction'' from the central region: $T_{n}, T_{ne},
T_{e}, T_{se}, T_{s}, T_{sw}, \mbox{ and } T_{w}$. Let $R = T_{nw}
\cup T_{ne} \cup T_{sw} \cup T_{se}$, the set of those squares of
$D$ whose orthogonals do not meet $U$.  Let $S = T_{n} \cup T_{e}
\cup T_{s} \cup T_{w}$, the set of those squares of $D$ having
exactly one orthogonal that meets $U$. Then $D$ is the disjoint
union of $R$, $S$, and $C$.

Since each column to the left of column $a$ contains at least one
square of $D$,
\begin{equation} \label{in1}
|T_{sw}| + |T_{w}| + |T_{nw}| \geq a-1.
\end{equation}
Similarly,
\begin{eqnarray}
\label{in2} |T_{se}| + |T_{e}| + |T_{ne}| & \geq & n-b, \\
\label{in3}|T_{sw}| + |T_{s}| + |T_{se}| & \geq & c-1, \\
\label{in4}|T_{nw}| + |T_{n}| + |T_{ne}| & \geq & m-d.
\end{eqnarray}
Adding inequalities (\ref{in1})-(\ref{in4}) and using the
def{\kern0pt}initions of $R$ and $S$ gives
\begin{equation} \label{in5}
2 |R| + |S| \geq m + n - 2 - (d-c) - (b-a).
\end{equation}
Each square in $R$ covers at most two squares of $E$, as the
square's orthogonals and one of its diagonals miss $E$. Each
square in $S$ covers at most six squares of $E$, as one of the
square's orthogonals misses $E$. Each square in $C$ covers eight
squares of $E$.  Since $D$ is a dominating set, $D$ covers all
squares of $E$, so
\begin{equation} \label{in6}
2 |R| + 6 |S| + 8 |C| \geq 2(d-c) + 2(b-a).
\end{equation}
Adding two times (\ref{in5}) to (\ref{in6}) gives
\begin{equation} \label{in7}
6 |R| + 8 |S| + 8 |C| \geq 2(m+n-2).
\end{equation}
Since $|D| = |R| + |S| + |C|$, adding $2 |R|$ to both sides of
(\ref{in7}) gives
\[
8 |D| \geq 2(m+n-2+|R|).
\]
Thus
\begin{equation} \label{ineq}
\gamma(Q_{m \times n}) = |D| \geq (m+n-2+|R|)/4,
\end{equation}
which implies the desired conclusion. \end{proof}

A diagram illustrating the proof for $Q_{10 \times 17}$ is given in Figure 4.

\begin{figure} \label{fig10x17sol8482}
\begin{center}
\begin{picture}(100,200)
\put(-195,-40){\resizebox{170mm}{!}{\rotatebox{0}{\includegraphics{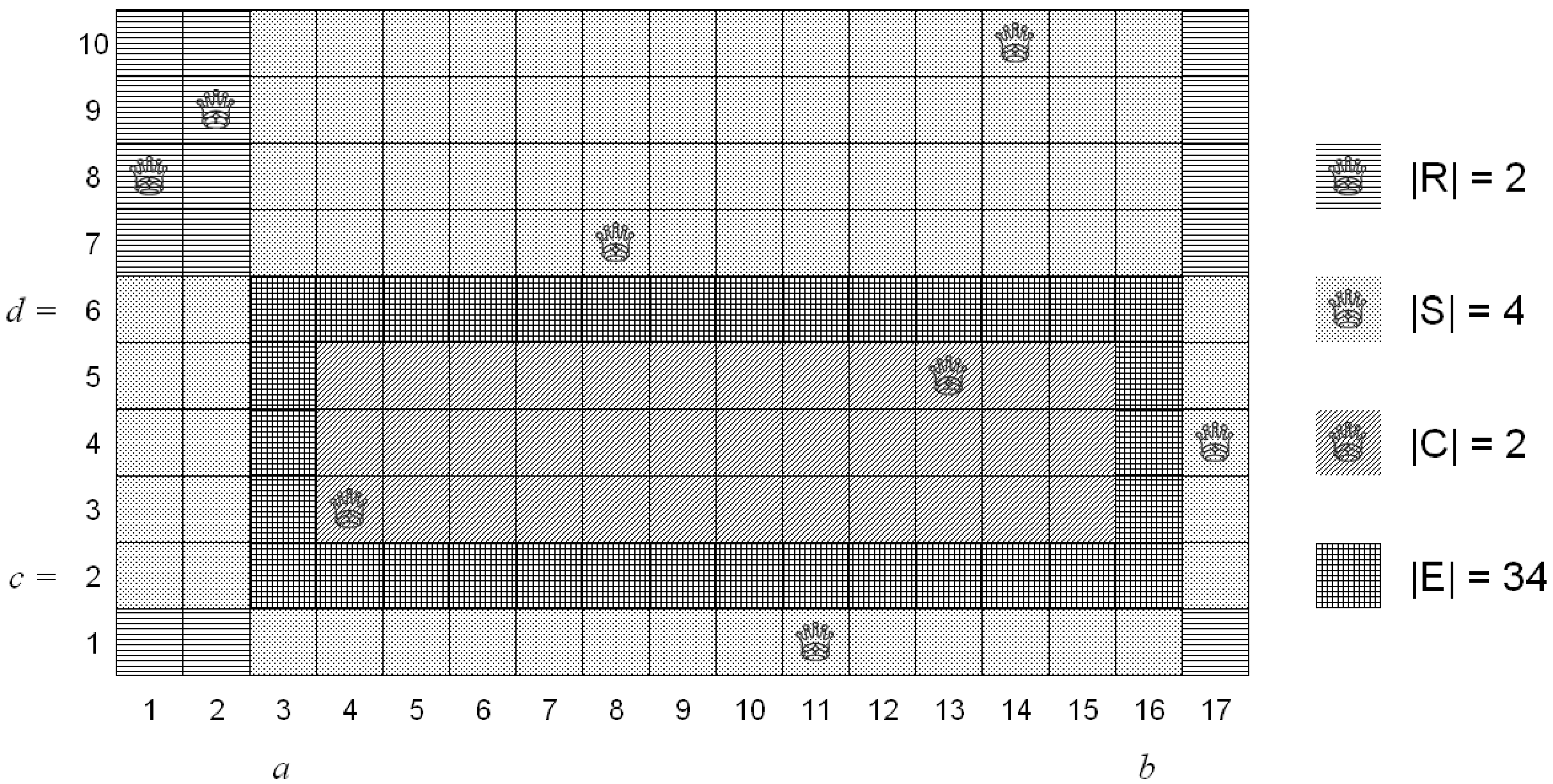}}}}
\end{picture}\\[12mm]
\caption{Illustration of the proof of Theorem \ref{main} with $Q_{10 \times
17}$: $R, S, C$ are the sets of queen squares having respectively $0, 1, \mbox{or } 2$ orthogonals meeting
the box of $D$, and $E$ is the box border of $D$.}
\end{center}
\end{figure}

There are 120 pairs $(m, n)$ satisfying $4 \leq m  \leq n \leq  18$.  Of these,
the bound (\ref{rvs2}) is achieved for 40 pairs (28 with $m \leq 6$),  for 76 pairs the bound
is exceeded by one and for the pairs $(12, 14), (13, 17), (14, 16)$, and $(15, 15)$ the bound is
exceeded by two.


We next explore when $\gamma(\q{m}{n}) = (m+n-2)/4$.  From Proposition \ref{prop1} it
follows that for any positive integer $m$, if $n = 3m+2$ then $\gamma(\q{m}{n}) = m$, and here $m =
(m+n-2)/4$. So we restrict to $n < 3m+2$.

\begin{cor} \label{coro}
Suppose $m \leq n < 3m+2$ and $\gamma(\q{m}{n}) = (m+n-2)/4$. Let $D$
be a minimum dominating set of \q{m}{n}.  Then $|D| \leq m-2$, each box border
square is covered exactly once by $D$, and $D$ is independent.
\end{cor}

\begin{proof}
From $n < 3m+2$ we have $|D| = (m+n-2)/4 \leq m-2$, so the box of $D$ is def{\kern0pt}ined.
Since $|D| = (m+n-2)/4$, in this setting we have equality in inequalities (\ref{in1})\mbox{--}(\ref{ineq}),
so each square of the box border $E$ is covered exactly once by $D$.
Thus any line meeting $E$
contains at most one square of $D$. Every square of $D$ must be diagonally adjacent to four squares of $E$, so
if any line containing a square of $D$ does not meet $E$, it is an orthogonal.
From (\ref{ineq}) we see that here the set $R$ of ``corner squares'' in $D$ is
empty, so every square of $D$ has at least one orthogonal meeting $E$; then
since (\ref{in1})--(\ref{in4}) are equations here, each orthogonal that misses $E$ contains exactly one
square of $D$.  Thus $D$ is independent.
\end{proof}

Rarely does a minimum dominating set cover each of its box border squares
uniquely; see Figure 3 for an example. We also note that the minimum dominating sets \#1-4
for \q{11}{12} have this property.
Each of these sets consists of a foursome centered at $(13/2, 13/2)$ plus the corner squares $(1, 1)$
and $(12, 1)$, so is not independent.

As mentioned
earlier, $\gamma(\q{n}{n}) = (n-1)/2$ is achieved only for $n = 3, 11$.  Considering Corollary \ref{coro}, we
suspect that the answer to the following question is no.

\begin{ques}  \label{wonderful}
Does $\gamqq = (m+n-2)/4$ with $m \leq n < 3m + 2$ occur, other than for $(m, n) = (3, 3)$ and $(11, 11)$?
\end{ques}

 We next extend the method of proof used in \cite{CO, RV, WE} for the lower bound (\ref{rvs}) to show that
the dimensions of the box of $D$ give a lower bound for $|D|$.

\begin{prop} \label{morebound}
Let $m \leq n$ and let $D$ be a dominating set of \qq\ of size at most $m-2$.  Let $m'$ be the number of rows
and $n'$ the number of columns of the box of $D$.  Then:

If $m' > n'$ then $|D| \geq \lceil \frac{n}{2} \rceil$;

If $m' \leq n'$ then $|D| \geq \left \lceil \frac{n - 1 - (n' - m')}{2} \right \rceil$.
\end{prop}

\begin{proof} Let $a, b, c, d$ be def{\kern0pt}ined as in the proof of Theorem \ref{main}.
(Then $m' = d-c + 1$ and $n' = b - a + 1$.) Since $m' \leq m \leq n$,
we may choose an integer $e$ such that $e$ through $e + m' - 2$ are indices of
columns of the board. Set $S = \{ (x, c), (x, d) \mbox{ : } e \leq x \leq e + m' -2 \}$ and
$P = \{ (x, y) \in D \mbox{ : } x < e \mbox{ or } x > e + m' - 2 \}$. Then no square is
diagonally adjacent to more than two squares of $S$ and no square of $P$ is orthogonally
adjacent to any square of $S$. As the $2(m'-1)$ squares of $S$ are covered by $D$, $2(m'-1) \leq 2|P| + 4(|D| - |P|)$, which implies
\begin{equation} \label{bounder}
|D| \geq \left \lceil \frac{m'-1+|P|}{2} \right \rceil.
\end{equation}
If $m' > n'$ then we can choose $e$ so that all columns that do not meet $S$ are occupied, so $|P| \geq n - (m'-1)$.
If $m' \leq n'$ we can choose $e$ so that $S$ is contained in the top and
bottom edges of $U$, and then
$|P| \geq  n - n'$. In both cases, (\ref{bounder}) implies the conclusion.
\end{proof}

As $\gamma(\q{n}{n}) = (n-1)/2$ only for $n = 3, 11$, we have $\gamma(\q{n}{n}) \geq \lceil n/2 \rceil$ for
all other positive integers $n$. There is much evidence that this lower bound is quite
good. Work from \cite{BM1,BMC,GW, KG, OW, WE, LB} reported in \cite{OW} shows that
for $n$ from 1 to 120, excluding 3 and 11, we have $\lceil n/2 \rceil \leq \gamma(\q{n}{n}) \leq \lceil n/2 \rceil + 1$.
In this range, $\gamma(\q{n}{n}) = \lceil n/2 \rceil$  is known for 46 values of $n$ and
$\gamma(\q{n}{n}) = \lceil n/2 \rceil +1$ is known for $n = 8, 14, 15, 16$. Also, $\gamma(\q{(4k+1)}{(4k+1)}) = 2k+1$ is known for $1 \leq k \leq 32$.

 For $m < n$, we have little evidence that the bound (\ref{rvs2}) is good.   We  were not able to use the
 methods of the proofs of Theorem \ref{main} and Proposition \ref{morebound}
to improve on this bound. Also,
a computer search using a greedy algorithm for some larger $m, n$ did not supply evidence about lower bounds for $\gamma(\qq)$.

The statement of Proposition \ref{morebound} leads one to consider the quantity
$n/2$.  We have checked that when $4 \leq m \leq n \leq 18$, $\gamma(\q{m}{n}) \geq \min\{m-1, \lfloor n/2
\rfloor -1 \}$. This bound and the bound (\ref{rvs2}) are close only when $m$
and $n$ are close.    So we ask the following.

\begin{ques} \label{whatlowerbound}
For $m, n$ with  $m \leq n$, what is a good general lower bound for $\gamma(\qq)$? In particular, is it true
that $\gamma(\qq) \geq \min \{ m-1, \lfloor n/2 \rfloor - 1 \}$?
\end{ques}

\section{Construction of dominating sets} \label{construct}

Given dimensions $m$ and $n$, we would like a general approach that would allow us to construct minimum
dominating sets of \q{m}{n}, or at least reasonably small dominating sets.  We have two dif{\kern0pt}f{\kern0pt}iculties to consider.

The f{\kern0pt}irst dif{\kern0pt}f{\kern0pt}iculty was just discussed:  in general we know the value of $\gamma(\qq)$ only approximately.

The second dif{\kern0pt}f{\kern0pt}iculty is that construction of a dominating set of \qq\ generally means specifying most or all
of the lines that the set is to occupy. There are some restrictions that the indices of the lines must satisfy, as we now describe.

Let $D = \{ (x_{i}, y_{i}) \mbox{ : } 1 \leq i \leq l \}$ be a set of $l$ squares of \qq\ that
occupies dif{\kern0pt}ference diagonals $(d_{i})_{i=1}^{l}$ and sum diagonals $(s_{i})_{i=1}^{l}$. Since the square \xy\
is on the \dd\ with index $y-x$ and the \sd\ with index $y+x$, summing over $D$ gives
\begin{equation} \label{linearcon}
\sum_{i=1}^{l} d_{i} = \sum_{i=1}^{l} y_{i} - \sum_{i=1}^{l} x_{i} \mbox{  and  }
\sum_{i=1}^{l} s_{i} = \sum_{i=1}^{l} y_{i} + \sum_{i=1}^{l} x_{i}.
\end{equation}

The Parallelogram Law $2x^{2} + 2y^{2} = (y-x)^{2} + (y+x)^{2}$ gives a quadratic constraint
\begin{equation} \label{plleleq}
2 \sum_{i=1}^{l} x_{i}^2 + 2 \sum_{i=1}^{l} y_{i}^2 = \sum_{i=1}^{l} d_{i}^2 + \sum_{i=1}^{l} s_{i}^2
\end{equation}
on the line indices.

In each of the two constructions given below, we will refer to lines that must be occupied for
domination as \emph{required} lines and other lines as \emph{auxiliary} lines.

Both constructions produce a number of minimum dominating sets, but neither can produce a dominating
set of \q{m}{n}\ of size less than $\lfloor n/2 \rfloor$. This is a little evidence for the possible bound
mentioned in Question \ref{whatlowerbound}.

\subsection{Domination by orthodox covers}

This idea generalizes \cite[Section 2]{UB}.  Let $D$ be a set of squares of \qq. If it is possible to place the
origin of the coordinate system so that every even column and every even row contains a square of $D$,
we will say $D$ is an \emph{orthodox} set.
That is, an orthodox set is one that occupies at least every other column and every other row of \qq.

Say that square \xy\ of \qq\ is \emph{even} if $x+y$ is even, \emph{odd} if $x+y$ is odd. We divide the even
squares of \qq\ into two classes: \xy\ is {\em even-even\/} if both $x$ and $y$ are even, {\em odd-odd\/} if
both are odd. If $D$ is an orthodox set and each odd-odd square of \qq\
is covered by
some square of $D$, we say $D$ is an \emph{orthodox cover}. For example, solution
\href{https://www.combinatorics.org/files/v26i4p45/07x11_5Q.html\#Solution10}{\#10}
for \q{7}{11} given in Table 1 is an orthodox cover; take the origin at
$(6, 3)$  to see this.

It is clear from the def{\kern0pt}inition that an orthodox cover dominates every even square of \qq, and since every odd
square of \qq\ is on one even-indexed orthogonal, all odd squares are also dominated: an orthodox cover is
a dominating set of \qq. Many orthodox covers appear in the
\href{https://www.combinatorics.org/ojs/index.php/eljc/article/view/v26i4p45/HTML}{appendix},
and are labeled there as such.

Since \qq\ has at least $\lfloor n/2 \rfloor$ even-indexed columns, an orthodox set on \qq\ has at
least $\lfloor n/2 \rfloor$ members.  Generally, we expect that most of the squares of $D$ will be
even-even, to help dominate the odd-odd squares diagonally.  When $n$ is considerably larger than $m$,
there are more possibilities of placing queens on odd squares that occupy even columns.  Also, it is possible
sometimes to achieve a dominating set of size less than $\lfloor n/2 \rfloor$ by a minor modif{\kern0pt}ication, as is
shown in  solution
\href{https://www.combinatorics.org/files/v26i4p45/07x12_5Q.html\#Solution1}{\#1}
for \q{7}{12} in Table 1. If the center of the square there labeled
$(6, 3)$
is taken to be the origin of the coordinate system, the dominating set shown misses being an orthodox
set only by not occupying the rightmost column. Thus the three odd-odd squares in that column are not covered
along their column, as they would be by an orthodox set.  But the queen on a dark square covers those three
squares, and the odd squares of its column, and thus completes a dominating set of size 5.

Minimum dominating set
\href{https://www.combinatorics.org/files/v26i4p45/12x16_8Q.html\#Solution147}{\#147}
for \q{12}{16}\ is an orthodox cover with a single queen on a dark square, at
$(6, 10)$. The squares covered only by this queen are the
dark squares in its column and $(1, 10)$. Replacing $(6, 10)$ with the white square $(6, 5)$
covers those squares and, adding a row 0 to the board, also the square $(1, 0)$.
In fact, the full set now covers all of row 0 and is thus a minimum dominating set of
\q{13}{16}; it is solution
\href{https://www.combinatorics.org/files/v26i4p45/13x16_8Q.html\#Solution15}{\#15} for \q{13}{16},
rotated by a half-turn.

There are many ways to create orthodox covers, and we will only give one example.  An approach is to
regard \q{m}{n}\ as the union of overlapping copies of \q{m}{m}; for odd $m$, this allows us to
use \cite[Theorem 1]{UB}, which gives suf{\kern0pt}f{\kern0pt}icient conditions for an orthodox set on \q{m}{m} to
be an orthodox cover.

\begin{exam}   \label{orthexam} An orthodox cover implying $\gamma(\q{13}{19}) \leq 10$.\\
We take the origin of the coordinate system to be the center of \q{13}{19}, and regard \q{13}{19}\ as the
union of two copies of \q{13}{13}, centered at $(\pm 3, 0)$.  From \cite[Theorem 1]{UB}, if we regard
the center of \q{13}{13}\ as the origin, an orthodox set on \q{13}{13}\ dominates if the set occupies the
sum and \dd s with indices in $\{-6, -2, 0, 2, 6 \}$. Asking this on both copies of \q{13}{13}, we wish to have
our orthodox set occupy the sum and \dd s which (on \q{13}{19}) have indices in $\{-6, -2, 0, 2, 6 \} \pm 3$, which is
$\{ \pm 1, \pm 3, \pm 5, \pm 9 \}$, so there will be two auxiliary \dd\ indices $d_{1}, d_{2}$ and two
auxiliary \sd\ indices $s_{1}, s_{2}$. The required column indices are $\pm 1, \pm 3, \pm 5, \pm 7, \pm 9$,
so there will be no auxiliary column indices.  The required row indices are $0, \pm 2, \pm 4, \pm 6$, so
there will be three auxiliary row indices $r_{1}, r_{2}, r_{3}$.

From (\ref{linearcon}) we have $d_{1} + d_{2} = r_{1} + r_{2} + r_{3} = s_{1} + s_{2}$ and (\ref{plleleq}) gives
$d_{1}^{2} + d_{2}^{2} + s_{1}^{2} + s_{2}^{2} = 420 + 2(r_{1}^{2} + r_{2}^{2} + r_{3}^{2})$.  We attempt
to f{\kern0pt}ind a solution with symmetry by a half-turn about the board center: this means
$r_{1} = 0, r_{2} = - r_{3}, d_{1} = -d_{2}, \mbox{ and } s_{1} = -s_{2}$.  Then the quadratic constraint simplif{\kern0pt}ies
to $d_{1}^{2} + s_{1}^{2} = 210 + 2r_{2}^{2}$, of which one solution is $d_{1} = 13, s_{1} = 7, r_{2} = 2$.
Now all lines are specif{\kern0pt}ied, and it is not dif{\kern0pt}f{\kern0pt}icult to f{\kern0pt}ind the solution $D = \{ \pm(9, 0), \pm(7, -6), \pm(5, 2), \pm(3, 2), \pm(1, -4) \}$.
\end{exam}

\subsection{Domination by centrally strong sets}

We begin by considering a board $C$ which is to be a central sub-board of a larger board $B$.
Say that $C$ has $m_{1}$ rows and $n_{1}$ columns, with $m_{1} \geq n_{1}$ and $m_{1}, n_{1}$ not both
even. It is convenient here to have the board squares of side length two, and place $C$ with its center at the origin.
Thus if, for example, $m_{1}$ is odd and $n_{1}$ is even, then each square has center \xy\ with $x$ an odd integer and $y$ an even one.

We then wish to choose a nonnegative integer $k$ and a set $D$ of squares of $B$ (actually, all or almost all in $C$)
such that $D$ contains at least one square from the extension of each orthogonal of $C$ to $B$, and $D$
contains exactly one square from the extension  of each \dd\ of $C$, except none from the highest $k$ and
the lowest $k$ extended \dd s; similarly for \sd s. Let
\begin{equation} \label{mnk}
m = m_{1} + 2n_{1} - 2k, \hspace{.3in}  n = 2m_{1} + n_{1} - 2k, \hspace{.3in} g = m_{1} + n_{1} - 2k - 1.
\end{equation}
Then $m \leq n$, and it is straightforward to verify that if $C$ is taken to be the central $m_{1} \times n_{1}$ sub-board of the $m \times n$ board $B$, then $D$ is a dominating set of \qq\ and $|D| = g$.
Such a $D$ will be called a \emph{centrally strong} set, as it generalizes the idea discussed for square
boards in \cite[page 234]{UB}. We note that our def{\kern0pt}inition requires each square of $D$ to have both
diagonals among the required ones, and thus both have indices of absolute value at
most $m_{1} + n_{1} - 2k - 2$, but this does not imply $D \subseteq C$.  If in fact $D \subseteq C$, we
say that $D$ is a \emph{strict} centrally strong set.

A number of strict centrally strong sets occur in the
\href{https://www.combinatorics.org/ojs/index.php/eljc/article/view/v26i4p45/HTML}{appendix},
and are labeled there as such.  We note that these sets can only occur when $m \leq n <
2m$; this follows from (\ref{mnk}) and the fact that since there will be $n_{1} - 2k - 1$
auxiliary row indices, this quantity is nonnegative.

One merit of this construction is that a single  centrally strong  $D$ gives an upper
bound for \gamqq\ for several   pairs $(m, n)$  since
$D$ is conf{\kern0pt}ined to a small  central  region of the $m \times n$ board, especially if $D$ is strict. For example, there
is a strict centrally strong set $D = \{ \pm(-5, 0)$, $\pm(-3, 4)$, $\pm(-1, 6)$, $\pm(1, 2)$, $\pm(3, 6) \}$
with $m_{1} = 9, n_{1} = 6, \mbox{ and } k=2$, and $|D| = 10$, which shows that $\gamqq \leq 10$
when $9 \leq m \leq 17 \mbox{ and } 6 \leq n \leq 20$.  For some of these pairs $(m, n)$, this bound is
poor, but for the six pairs with $m + n \geq 35$, combining with the bound (\ref{rvs2}) gives
$9 \leq \gamqq \leq 10$, and 10 is a useful upper bound for some of the smaller boards also.

The simplest centrally strong sets occur with $m_{1} \geq 1$, $n_{1} = 1$ and $k = 0$, where we get
$m_{1}$ queens occupying all squares of the $m_{1} \times 1$ board $C$, and the following bound (which
we have stated in terms of $m = m_{1} + 2$). For $3 \leq m \leq 10$ at least, this bound gives the exact
value of $\gamma(\q{m}{(2m-3)})$.

\begin{prop}  \label{n1equal1}
For $m \geq 3$, $\gamma(\q{m}{(2m-3)}) \leq m-2$.
\end{prop}

We next consider the ef{\kern0pt}fect of (\ref{linearcon}) and (\ref{plleleq}) on the search for centrally strong sets.
Symmetry and the requirement that each \dd\ contains exactly one square of $D$ imply that the sum of the \dd\ indices of $D$ is zero.
Similarly the sum of the \sd\ indices of $D$ is zero, and then (\ref{linearcon}) implies that $\sum_{(x, y) \in D} x = 0$ and
$\sum_{(x, y) \in D} y = 0$. As we require a centrally strong set to occupy all (extended) columns of the sub-board, we regard the
$n_{1}$ indices of these columns as required column indices; by symmetry their sum is zero. As $C$ has $n_{1}$ columns and
$g$ occupied squares, there will be $g - n_{1} = m_{1}-2k-1$ auxiliary column indices, each having parity opposite to that of $n_{1}$.
Since $\sum_{(x, y) \in D} x = 0$ and all required column indices sum to zero, so do the auxiliary column indices. Similarly
there will be $m_{1}$ required row indices and  $g - m_{1} = n_{1} - 2k - 1$ auxiliary row indices, with sum zero, each
having parity opposite to that of $m_{1}$.  If $D$ is strict, then all indices of occupied columns have absolute
value at most $n_{1} - 1$ and all indices of occupied rows have absolute value at most $m_{1} - 1$.
(We have required that $m_{1}, n_{1}$ not both be even because if they were, there would be an odd number of auxiliary row indices,
each odd, so their sum could not be even, thus not zero.)

Using the identities $\sum_{i=1}^{j} (2i-1)^{2} = \binom{2j+1}{3}$ and $\sum_{i=1}^{j} (2i)^{2} = \binom{2j+2}{3}$,
we see that the sum of the squares of the indices of all occupied diagonals of $C$ is
$4 \binom{g+1}{3}$, the sum of the squares of the required column indices is $2 \binom{n_{1}+1}{3}$ and the sum
of the squares of the required row indices is $2 \binom{m_{1}+1}{3}$. Letting $\sum_{orth}$ denote the sum of the
squares of the auxiliary column indices and auxiliary row indices, the quadratic constraint (\ref{plleleq}) gives
\begin{equation}  \label{quadcstrong}
\sum_{orth} = 2\left[ \binom{g+1}{3} - \binom{m_{1} + 1}{3} - \binom{n_{1} + 1}{3} \right].
\end{equation}

Combined with Proposition \ref{n1equal1}, part (a) of the following proposition shows how small a centrally strong set can be. In parts (b) and (c), we limit the values of $m_{1}, n_{1}, k$ that need be considered when constructing centrally strong sets.

We say that a value of $k$ for which there exists a centrally strong set on \q{m_{1}}{n_{1}}\ is \emph{feasible} for $(m_{1}, n_{1})$.

\begin{prop}  \label{atbesthalfn}
(a) For any centrally strong set $D$ with $n_{1} > 1$, $|D| \geq n/2$.\\
(b) For any $(m_{1}, n_{1})$, it is only necessary to use the largest feasible $k$ to determine all upper bounds
for \gamqq\ implied by centrally strong sets from $(m_{1}, n_{1})$.\\
(c) If $k$ is feasible for $(m_{1}, n_{1})$ and $k+1$ is feasible for $(m_{1}, n_{1}+2)$, the latter gives the more useful result.
\end{prop}

\begin{proof}
(a): As the number $n_{1}-2k-1$ of auxiliary row indices is nonnegative, $n_{1} \geq 2k+1$.
If $n_{1} = 2k+1$ then $g = m_{1}$ by (\ref{mnk}), and then the fact that the right side of (\ref{quadcstrong})
is nonnegative implies $n_{1}=1$ and $k=0$, the situation of Proposition \ref{n1equal1}.  Thus for $k \geq 1$
we have $n_{1} \geq 2k+2$, which by (\ref{mnk}) is equivalent to $|D| \geq n/2$.

(b): Suppose for some integer $h>0$ that both $k$ and $k-h$ are feasible for $(m_{1}, n_{1})$. Then
the triple $m_{1}, n_{1}, k$ gives a dominating set $D$ of size $g$ on \q{m}{n}, where $m, n, g$ are
determined by (\ref{mnk}), and similarly the triple $m_{1}, n_{1}, k-h$ gives a dominating set $D'$ of
size $g+2h$ on \q{(m+2h)}{(n+2h)}. However, by repeating $2h$ times the process of adding an edge row
 and edge column to the board and the new corner  square to the dominating set, we can construct from $D$ a dominating
set  of  \q{(m+2h)}{(n+2h)}\ of the same size as $D'$.

(c): Using (\ref{mnk}), if $m_{1}, n_{1}, k$ gives a dominating set of size $g$ for \qq, then
$m_{1}, n_{1}+2, k+1$ gives a dominating set of size $g$ for \q{(m+2)}{n}.
\end{proof}

\begin{exam}  \label{13by16}  A centrally strong set implying  $\gamma(\q{13}{16}) \leq 8$.

Let $m_{1} = 7$ and $n_{1} = 4$, and $k=1$.  Then a strict centrally strong set $D$ is to have one auxiliary
row index, which from $\sum_{(x, y) \in D} y = 0$ must be zero, and two auxiliary column indices, say
$c_{1}, c_{2}$, each in $\{ -3, -1, 1, 3 \}$. From $\sum_{(x, y) \in D} x = 0$ we see $c_{2} = -c_{1}$ and
from (\ref{quadcstrong}) we have $c_{1}^{2} + c_{2}^{2} = 18$, so we can take $c_{1} = 3$ and $c_{2} = -3$.
We then easily obtain $D = \{ \pm(1, -6), \pm(3, 4), \pm(3, 0), \pm(3, -2) \}$;
see solution \href{https://www.combinatorics.org/files/v26i4p45/13x16_8Q.html\#Solution23}{\#23} for $\q{13}{16}$.
(Recall that board squares have edge length two,  column indices are even integers, and row indices
are odd integers here.) Using (\ref{mnk}) this gives $\gamma(\q{13}{16}) \leq 8$ (and equality holds by
our computer search).
\end{exam}

We give two inf{\kern0pt}inite families of strict centrally strong sets,
each including a minimum dominating set found by von Szily
\cite{V1,V2}.

\begin{exam} Strict centrally strong sets for $n_{1} = 5, k = 1$ and odd $m_{1} \geq 5$, and for $n_{1} = 7, k=2$ and odd $m_{1} \geq 7$.

In our approach described above, all orthogonal indices would be even here; we have divided by two, thus
returning to a board with squares of edge length one.

For $n_{1} = 5, k = 1, \mbox{ and } m_{1} \equiv 1 \beem{4}$, $D$ consists of $\pm(-1, \frac{m_{1}-1}{2})$,
$(0, 0)$, and $\pm(0, 2i)$ and $\pm(2, \frac{m_{1} + 5}{2} - 4i)$ for $1 \leq i \leq \frac{m_{1} - 1}{4}$.
With $m_{1} = 5$, this gives a minimum dominating set of \q{13}{13}\ found by
von Szily \cite{V1};
see also solution \href{https://www.combinatorics.org/files/v26i4p45/13x13_7Q.html\#Solution41}{\#41}.

For $n_{1} = 5, k = 1, \mbox{ and } m_{1} \equiv -1 \beem{4}$, $D$ consists of $\pm(\pm 1, \frac{m_{1}-1}{2})$, $\pm(-1, \frac{m_{1}-3}{2})$, $(0, 0)$, $\pm(0, 2i)$ for $1 \leq i \leq \frac{m_{1} - 7}{4}$, and $\pm(2, \frac{m_{1} + 3}{2} - 4i)$ for $1 \leq i \leq \frac{m_{1} - 3}{4}$.

These sets show that for $i \geq 3$, if $2i-1 \leq m \leq 2i+7$ and $5 \leq n \leq 4i+1$, then $\gamma(\q{m}{n}) \leq 2i+1$. \vspace{.1in}

Now let $n_{1} = 7$.

For $m_{1} = 7$, let $D = \{ i(1, 2) + j(2, -1) \mbox{ : } -1 \leq i, j \leq 1 \}$. This gives a minimum dominating
set of \q{17}{17}\ found by von Szily \cite{V2}; see also solution
 \href{https://www.combinatorics.org/files/v26i4p45/17x17_9Q.html\#Solution21}{\#21}.

For $m_{1} = 9$, let $D = \{ (0, 0), \pm(1, 4), \pm(2, -3), \pm(1, 2) + j(2, -1) \mbox{ : } -1 \leq j \leq 1 \}$. This
gives $\gamma(\q{19}{21}) \leq 11$, which is the best we know.

The following complicated description of a placement is the result of unifying four cases depending on the
residue of $m_{1}$ modulo 8. Any odd $m_{1} \geq 11$ has a unique expression
$m_{1} = 11 + 2(l_{1} + l_{2})$ with $l_{1}$ an integer and either $l_{2} = l_{1}$ or $l_{2} = l_{1} + 1$.
(Here $l_{1} = \lfloor (m_{1} - 11)/4 \rfloor$ and $l_{2} = \lceil (m_{1} - 11)/4 \rceil$.)

Start with $(0, 0), \pm(1, 2), \pm(2, -3), \pm(3, -1), \pm((-1)^{l_{1}}, -2l_{1}-5), \pm((-1)^{l_{2}+1}, -2l_{2}-4)$.
Add $\pm(2, 4j)$ and $\pm(2, 4j+1)$ for $1 \leq j \leq \lceil l_{2}/2 \rceil$, and add $\pm(2, -4j-2)$ and
$\pm(2, -4j-3)$ for $1 \leq j \leq \lfloor l_{2}/2 \rfloor$. If $l_{2} = l_{1}$ then add $\pm(2, (-1)^{l_{1}}(2l_{1}+4))$.

These sets show that for $i \geq 4$, if $2i-1 \leq m \leq 2i+9$ and $7 \leq n \leq 4i+1$,
then $\gamma(\q{m}{n}) \leq 2i+1$.
\end{exam}

Above we have described two approaches to the construction of dominating sets.
In both, once a set of lines to be occupied by the dominating set is specif{\kern0pt}ied,
it is necessary to see whether one can f{\kern0pt}ind such a dominating  set.  A fast backtrack search idea of Hitotumatu and Noshita \cite{HN},
explained and amplif{\kern0pt}ied by Knuth \cite{KN}, was used by \"{O}sterg{\aa}rd and Weakley \cite{OW} to f{\kern0pt}ind
values and bounds of $\gamma(\q{n}{n})$ up to $n = 120$. This approach and also the algorithm of
Neuhaus \cite{NE} can be applied to rectangular boards as well.  But as mentioned, neither of our constructions
can produce a dominating set of size less than $\lfloor n/2 \rfloor$ for $Q_{m \times n}$ (with $m \leq n$).  Thus a
resolution of Question \ref{whatlowerbound} would be needed to determine whether extensive search based on these
constructions is useful.

The complexity of computing minimum dominating set of queens is
another open question \cite[Section 5]{FE}.
Backtracking algorithms, dynamic programming, and treewidth technique
are analyzed extensively by Fernau \cite[Sections 2-4]{FE}.

Applications of backtracking algorithms to a variety of domination
problems are studied in the doctoral dissertation of Bird
\cite{Bird}; in particular, he examines how recursive backtracking
search can be split among multiple processes by partitioning the
search tree.  We give some of his results on queens at the end of
the next section.

\section{Independent domination}

We have calculated the independent domination number $i(Q_{m \times n}), \, 4 \leq m  \leq n \leq  18$,
 as shown in the table. Each table entry is linked to  some   minimum independent
 dominating sets.

In these ranges for $m$ and $n$, monotonicity fails twice:
$i(Q_{8 \times 11}) = 6 > 5 = i(Q_{9 \times 11}) = i(Q_{10 \times 11}) = i(Q_{11 \times 11})$,
and $i(Q_{11 \times 18}) = 9 > 8 = i(Q_{12 \times 18})$.  The f{\kern0pt}irst instance is
essentially the same failure as for $\gamma(Q_{m \times n})$.  The second
is similar in that the only (up to symmetry) independent dominating set of size 8 for $Q_{12 \times 18}$
does not f{\kern0pt}it on $Q_{11 \times 18}$.

From the def{\kern0pt}initions it is clear that $\gamma(Q_{m \times n}) \leq i(Q_{m \times
n})$, and this appears to be an excellent lower bound for $i(Q_{m \times n})$.
In the table, we have highlighted  the entries where
these two numbers are unequal.  We know of no case where
$\gamma(Q_{m \times n}) + 1 < i(Q_{m \times n})$.

\begin{center}
\begin{tabular}{|c||c|c|c|c|c|c|c|c|c|c|c|c|c|c|c|}
\hline
 $n \diagdown m$ & \hspace*{0.1mm} 4 \hspace*{0.1mm}
                 & \hspace*{0.1mm} 5 \hspace*{0.1mm}
                 & \hspace*{0.1mm} 6 \hspace*{0.1mm}
                 & \hspace*{0.1mm} 7 \hspace*{0.1mm}
                 & \hspace*{0.1mm} 8 \hspace*{0.1mm}
                 & \hspace*{0.1mm} 9 \hspace*{0.1mm} & 10 & 11 & 12 & 13 & 14 & 15 & 16 & 17 & 18  \\
\hline
\hline
               4 &  \cellcolor{lightgray} \href{https://www.combinatorics.org/files/v26i4p45/04x04_3Qi.html}{\emph{3}} & & & & & & & & & & & & & &   \\
\hline
               5 &  \cellcolor{lightgray} \href{https://www.combinatorics.org/files/v26i4p45/04x05_3Qi.html}{3} &
                    \href{https://www.combinatorics.org/files/v26i4p45/05x05_3Qi.html}{\emph{3}} & & & & & & & & & & & & &  \\
\hline
               6 &  \href{https://www.combinatorics.org/files/v26i4p45/04x06_3Qi.html}{3} &
                    \href{https://www.combinatorics.org/files/v26i4p45/05x06_3Qi.html}{3} &
                    \cellcolor{lightgray} \href{https://www.combinatorics.org/files/v26i4p45/06x06_4Qi.html}{\emph{4}} & & & & & & & & & & & & \\
\hline
               7 &  \href{https://www.combinatorics.org/files/v26i4p45/04x07_3Qi.html}{3} &
                    \cellcolor{lightgray} \href{https://www.combinatorics.org/files/v26i4p45/05x07_4Qi.html}{4} &
                    \href{https://www.combinatorics.org/files/v26i4p45/06x07_4Qi.html}{4} &
                    \href{https://www.combinatorics.org/files/v26i4p45/07x07_4Qi.html}{\emph{4}} & & & & & & & & & & & \\
\hline
               8 & \cellcolor{lightgray}  \href{https://www.combinatorics.org/files/v26i4p45/04x08_4Qi.html}{4} &
                    \href{https://www.combinatorics.org/files/v26i4p45/05x08_4Qi.html}{4} &
                    \href{https://www.combinatorics.org/files/v26i4p45/06x08_4Qi.html}{4} &
                    \href{https://www.combinatorics.org/files/v26i4p45/07x08_5Qi.html}{5} &
                    \href{https://www.combinatorics.org/files/v26i4p45/08x08_5Qi.html}{\emph{5}} & & & & & & & & & & \\
\hline
               9 &  {4} &
                    \href{https://www.combinatorics.org/files/v26i4p45/05x09_4Qi.html}{4} &
                    \href{https://www.combinatorics.org/files/v26i4p45/06x09_4Qi.html}{4} &
                    \href{https://www.combinatorics.org/files/v26i4p45/07x09_5Qi.html}{5} &
                    \href{https://www.combinatorics.org/files/v26i4p45/08x09_5Qi.html}{5} &
                    \href{https://www.combinatorics.org/files/v26i4p45/09x09_5Qi.html}{\emph{5}}  & & & & & & & & & \\
\hline
              10 &  {4} &
                    \href{https://www.combinatorics.org/files/v26i4p45/05x10_4Qi.html}{4} &
                    \href{https://www.combinatorics.org/files/v26i4p45/06x10_4Qi.html}{4} &
                    \href{https://www.combinatorics.org/files/v26i4p45/07x10_5Qi.html}{5} &
                    \href{https://www.combinatorics.org/files/v26i4p45/08x10_5Qi.html}{5} &
                    \href{https://www.combinatorics.org/files/v26i4p45/09x10_5Qi.html}{5} &
                    \href{https://www.combinatorics.org/files/v26i4p45/10x10_5Qi.html}{\emph{5}}  & & & & & & & & \\
\hline
              11 &  {4} &
                    \href{https://www.combinatorics.org/files/v26i4p45/05x11_4Qi.html}{4} &
                    \href{https://www.combinatorics.org/files/v26i4p45/06x11_5Qi.html}{5} &
                    \href{https://www.combinatorics.org/files/v26i4p45/07x11_5Qi.html}{5} &
                    \href{https://www.combinatorics.org/files/v26i4p45/08x11_6Qi.html}{\textbf{6}} &
                    \href{https://www.combinatorics.org/files/v26i4p45/09x11_5Qi.html}{5} &
                    \href{https://www.combinatorics.org/files/v26i4p45/10x11_5Qi.html}{5} &
                    \href{https://www.combinatorics.org/files/v26i4p45/11x11_5Qi.html}{\emph{5}}   & & & & & &  &\\
\hline
              12 &  {4} &
                    \href{https://www.combinatorics.org/files/v26i4p45/05x12_4Qi.html}{4} &
                    \href{https://www.combinatorics.org/files/v26i4p45/06x12_5Qi.html}{5} &
                    \href{https://www.combinatorics.org/files/v26i4p45/07x12_5Qi.html}{5} &
                    \href{https://www.combinatorics.org/files/v26i4p45/08x12_6Qi.html}{6}      &
                    \href{https://www.combinatorics.org/files/v26i4p45/09x12_6Qi.html}{6} &
                    \href{https://www.combinatorics.org/files/v26i4p45/10x12_6Qi.html}{6}  &
                    \href{https://www.combinatorics.org/files/v26i4p45/11x12_6Qi.html}{6}  &
                    \cellcolor{lightgray} \href{https://www.combinatorics.org/files/v26i4p45/12x12_7Qi.html}{\emph{7}}  & & & & & &  \\
\hline
              13 &  4 &
                    {5}  &
                    \href{https://www.combinatorics.org/files/v26i4p45/06x13_5Qi.html}{5} &
                    \href{https://www.combinatorics.org/files/v26i4p45/07x13_6Qi.html}{6}  &
                    \href{https://www.combinatorics.org/files/v26i4p45/08x13_6Qi.html}{6} &
                    \href{https://www.combinatorics.org/files/v26i4p45/09x13_6Qi.html}{6} &
                    \href{https://www.combinatorics.org/files/v26i4p45/10x13_7Qi.html}{7}  &
                    \href{https://www.combinatorics.org/files/v26i4p45/11x13_7Qi.html}{7}  &
                    \href{https://www.combinatorics.org/files/v26i4p45/12x13_7Qi.html}{7}   &
                    \href{https://www.combinatorics.org/files/v26i4p45/13x13_7Qi.html}{\emph{7}} & & & &  &  \\
\hline
              14 &  4 &
                    {5} &
                    {6} &
                    \href{https://www.combinatorics.org/files/v26i4p45/07x14_6Qi.html}{6}  &
                    \href{https://www.combinatorics.org/files/v26i4p45/08x14_6Qi.html}{6} &
                    \href{https://www.combinatorics.org/files/v26i4p45/09x14_6Qi.html}{6} &
                    \href{https://www.combinatorics.org/files/v26i4p45/10x14_7Qi.html}{7}  &
                    \href{https://www.combinatorics.org/files/v26i4p45/11x14_7Qi.html}{7}   &
                    \href{https://www.combinatorics.org/files/v26i4p45/12x14_8Qi.html}{8}    &
                    \href{https://www.combinatorics.org/files/v26i4p45/13x14_8Qi.html}{8}   &
                    \href{https://www.combinatorics.org/files/v26i4p45/14x14_8Qi.html}{\emph{8}} & & &  &             \\
\hline
              15 &  4 &
                    {5} &
                    {6} &
                    \href{https://www.combinatorics.org/files/v26i4p45/07x15_6Qi.html}{6} &
                    \cellcolor{lightgray} \href{https://www.combinatorics.org/files/v26i4p45/08x15_7Qi.html}{7} &
                    \href{https://www.combinatorics.org/files/v26i4p45/09x15_7Qi.html}{7} &
                    \href{https://www.combinatorics.org/files/v26i4p45/10x15_7Qi.html}{7} &
                    \href{https://www.combinatorics.org/files/v26i4p45/11x15_7Qi.html}{7} &
                    \href{https://www.combinatorics.org/files/v26i4p45/12x15_8Qi.html}{8} &
                    \href{https://www.combinatorics.org/files/v26i4p45/13x15_8Qi.html}{8} &
                    \cellcolor{lightgray} \href{https://www.combinatorics.org/files/v26i4p45/14x15_9Qi.html}{9} &
                    \href{https://www.combinatorics.org/files/v26i4p45/15x15_9Qi.html}{\emph{9}} & & &             \\
\hline
              16 &  4 &
                    5 &
                    {6} &
                    \href{https://www.combinatorics.org/files/v26i4p45/07x16_6Qi.html}{6} &
                    \href{https://www.combinatorics.org/files/v26i4p45/08x16_7Qi.html}{7} &
                    \href{https://www.combinatorics.org/files/v26i4p45/09x16_7Qi.html}{7} &
                    \href{https://www.combinatorics.org/files/v26i4p45/10x16_7Qi.html}{7} &
                    \href{https://www.combinatorics.org/files/v26i4p45/11x16_8Qi.html}{8} &
                    \href{https://www.combinatorics.org/files/v26i4p45/12x16_8Qi.html}{8} &
                    \href{https://www.combinatorics.org/files/v26i4p45/13x16_8Qi.html}{8} &
                    \href{https://www.combinatorics.org/files/v26i4p45/14x16_9Qi.html}{9} &
                    \href{https://www.combinatorics.org/files/v26i4p45/15x16_9Qi.html}{9} &
                    \href{https://www.combinatorics.org/files/v26i4p45/16x16_9Qi.html}{\emph{9}} &  &           \\
\hline
              17 &  4 &
                    5 &
                    6 &
                    {7} &
                    \href{https://www.combinatorics.org/files/v26i4p45/08x17_7Qi.html}{7} &
                    \href{https://www.combinatorics.org/files/v26i4p45/09x17_7Qi.html}{7}  &
                    \href{https://www.combinatorics.org/files/v26i4p45/10x17_8Qi.html}{8} &
                    \href{https://www.combinatorics.org/files/v26i4p45/11x17_8Qi.html}{8} &
                    \href{https://www.combinatorics.org/files/v26i4p45/12x17_8Qi.html}{8} &
                    \href{https://www.combinatorics.org/files/v26i4p45/13x17_9Qi.html}{9} &
                    \href{https://www.combinatorics.org/files/v26i4p45/14x17_9Qi.html}{9} &
                    \href{https://www.combinatorics.org/files/v26i4p45/15x17_9Qi.html}{9} &
                    \href{https://www.combinatorics.org/files/v26i4p45/16x17_9Qi.html}{9} &
                    \href{https://www.combinatorics.org/files/v26i4p45/17x17_9Qi.html}{\emph{9}}  &            \\
\hline
              18 &  4 &
                    5 &
                    6 &
                    {7} &
                    \href{https://www.combinatorics.org/files/v26i4p45/08x18_7Qi.html}{7} &
                    \href{https://www.combinatorics.org/files/v26i4p45/09x18_8Qi.html}{8} &
                    \href{https://www.combinatorics.org/files/v26i4p45/10x18_8Qi.html}{8} &
                    \cellcolor{lightgray} \href{https://www.combinatorics.org/files/v26i4p45/11x18_9Qi.html}{\textbf{9}} &
                    \href{https://www.combinatorics.org/files/v26i4p45/12x18_8Qi.html}{8} &
                    \href{https://www.combinatorics.org/files/v26i4p45/13x18_9Qi.html}{9} &
                    \href{https://www.combinatorics.org/files/v26i4p45/14x18_9Qi.html}{9} &
                    \href{https://www.combinatorics.org/files/v26i4p45/15x18_9Qi.html}{9} &
                    \cellcolor{lightgray} \href{https://www.combinatorics.org/files/v26i4p45/16x18_10Qi.html}{10} &
                    \cellcolor{lightgray} \href{https://www.combinatorics.org/files/v26i4p45/17x18_10Qi.html}{10} &
                    \cellcolor{lightgray} \href{https://www.combinatorics.org/files/v26i4p45/18x18_10Qi.html}{\emph{10}}    \\

\hline
\end{tabular} \\[1mm]
Table 2: Values of independent domination number $i(Q_{m \times n}), \, 4 \leq m  \leq n \leq  18$ (\href{http://oeis.org/A299029}{OEIS A299029}).
Highlighted cells indicate where $\gamma \neq i$.
\end{center}

 Bird \cite[Chapter 5]{Bird} reports the new values
$i(Q_{n \times n}) = \gamma(Q_{n \times n}) = (n/2) + 1$ for $n = 20, 22, 24$,
$i(Q_{19 \times 19}) = i(Q_{21 \times 21}) = 11$, and $i(Q_{23 \times 23}) = 13$.
He also gives the number of minimum dominating sets and the number of minimum
independent dominating sets, up to equivalence, for $Q_{n \times n}$ up to $n = 18$.

\section*{Acknowledgements}
The authors thank the anonymous reviewers for their substantial comments.
Research of the f{\kern0pt}irst author was supported in part by the
Hungarian Scientif{\kern0pt}ic Research Fund (OTKA), grant no.~K111797,
 and by the J\'anos Bolyai Research Fellowship of the
Hungarian Academy of Sciences, grant no.~BO/00154/16,
 and by the \'{U}NKP-18-4-BCE-90 Bolyai+ New National Excellence Program of the Ministry of Human Capacities, Hungary.

Parts of the computations were carried out on NIIF Supercomputer based in Hungary
at NIIF National Information Infrastructure Development Institute.

\end{document}